\documentclass[smallextended,natbib,runningheads]{svjour3}
\smartqed  % flush right qed marks, e.g. at end of proof
\usepackage{graphicx}
\usepackage{amsmath}
 \usepackage{marvosym}
 \usepackage{mathptmx}
 \usepackage{times,bm}
 \usepackage{bbold}
 \usepackage[LY1]{fontenc}
\usepackage{latexsym}
\usepackage{amsmath}
\usepackage{amssymb}
\usepackage{dsfont}
\usepackage{upgreek}
\usepackage{breqn}
\usepackage{enumerate}
\usepackage{enumitem}
\usepackage[dvipsnames,usenames]{color}
\usepackage{extarrows}
\RequirePackage[colorlinks,citecolor=Blue,linkcolor=blue,urlcolor=red]{hyperref}
\definecolor{Blue}{rgb}{0,0,.4}
\usepackage{caption}
\usepackage{graphicx}
\usepackage[font+=scriptsize,skip=0ex]{subcaption}
\usepackage{pgfplots}
\usepackage{environ}
\usepackage{tikz-3dplot}
\usetikzlibrary{arrows}
\usetikzlibrary{3d}
\tikzset{>=latex}
\usetikzlibrary{hobby}
\DeclareMathAlphabet{\mathpzc}{OT1}{pzc}{m}{it}
   \font\twelvebm                       = cmmib10 at 12truept
   \font\tenbm                          = cmmib10 at 10truept
   \font\sevenbm                        = cmmib10 at 7truept
  at 10truept
 at 10truept
 at 10truept
 at 10truept
 at 10truept
\mathchardef \al   ="010B
  \renewcommand{\alpha}{\textrm{\scalebox{.883}{$\al$}}}
\mathchardef \bet  ="010C
  \renewcommand{\beta}{\textrm{\scalebox{.882}{$\bet$}}}
\mathchardef \gam  ="010D
  \renewcommand{\gamma}{\textrm{\scalebox{.883}{$\gam$}}}
\mathchardef \del  ="010E
  \renewcommand{\delta}{\textrm{\scalebox{.882}{$\del$}}}
\mathchardef \eps  ="010F
  \renewcommand{\epsilon}{\textrm{\scalebox{.883}{$\eps$}}}
\mathchardef \zet  ="0110
  \renewcommand{\zeta}{\textrm{\scalebox{.882}{$\zet$}}}
\mathchardef \et   ="0111
  \renewcommand{\eta}{\textrm{\scalebox{.883}{$\et$}}}
\mathchardef \thet ="0112
  \renewcommand{\theta}{\textrm{\scalebox{.93}{$\thet$}}}
\mathchardef \iot  ="0113
  \renewcommand{\iota}{\textrm{\scalebox{.883}{$\iot$}}}
\mathchardef \kapp ="0114
  \renewcommand{\kappa}{\textrm{\scalebox{.883}{$\kapp$}}}
\mathchardef \lambd="0115
  \renewcommand{\lambda}{\textrm{\scalebox{.883}{$\lambd$}}}
\mathchardef \m    ="0116
  \renewcommand{\mu}{\textrm{\scalebox{.882}{$\m$}}}
\mathchardef \n    ="0117
  \renewcommand{\nu}{\textrm{\scalebox{.883}{$\n$}}}
\mathchardef \ksi  ="0118
  \renewcommand{\xi}{\textrm{\scalebox{.882}{$\ksi$}}}
\mathchardef \pii  ="0119
  \renewcommand{\pi}{\textrm{\scalebox{.883}{$\pii$}}}
\mathchardef \ro   ="011A
  \renewcommand{\rho}{\textrm{\scalebox{.883}{$\ro$}}}
\mathchardef \sig  ="011B
  \renewcommand{\sigma}{\textrm{\scalebox{.883}{$\sig$}}}
\mathchardef \ta   ="011C
  \renewcommand{\tau}{\textrm{\scalebox{.883}{$\ta$}}}
\mathchardef \upsil="011D
  \renewcommand{\upsilon}{\textrm{\scalebox{.883}{$\upsil$}}}
\mathchardef \fii  ="011E
  \renewcommand{\phi}{\textrm{\scalebox{.882}{$\fii$}}}
\mathchardef \xii  ="011F
  \renewcommand{\chi}{\textrm{\scalebox{.883}{$\xii$}}}
\mathchardef \psii ="0120
  \renewcommand{\psi}{\textrm{\scalebox{.883}{$\psii$}}}
\mathchardef \omeg ="0121
  \renewcommand{\omega}{\textrm{\scalebox{.883}{$\omeg$}}}
\mathchardef \vthet="0123
  \renewcommand{\vartheta}{\textrm{\scalebox{.93}{$\vthet$}}}
\mathchardef \vpi  ="0125
  \renewcommand{\varpi}{\textrm{\scalebox{.883}{$\vpi$}}}
\mathchardef \vsig  ="0126
  \renewcommand{\varsigma}{\textrm{\scalebox{.883}{$\vsig$}}}
  \mathchardef \vphi  ="0127
  \renewcommand{\varphi}{\textrm{\scalebox{.882}{$\vphi$}}}
\newcommand{\R}{\mathds{R}}

\newcommand{\E}{\operatorname{\mathds{E}}}
\newcommand{\Var}{\operatorname{\mathsf{Var}}}
\newcommand{\Cov}{\mathsf{Cov}}
\newcommand{\CF}{\mathpzc{F}}
\newcommand{\CC}{\mathpzc{C}}
\renewcommand{\Pr}{\operatorname{\mathds{P}}}
\newcommand{\ds}{\displaystyle}
\newcommand{\ud}[1]{\operatorname{d\textrm{${#1}$}}}
\newcommand{\prob}{\xlongrightarrow{\mathrm{p}}}%\newcommand{\prob}{\stackrel{\mathrm{p}}{\longrightarrow}}
\newcommand{\distr}{\xlongrightarrow{\mathrm{d}}}%\newcommand{\distr}{\stackrel{\mathrm{d}}{\longrightarrow}}
\newcommand{\law}{\xlongequal{\mathrm{d}}}%\newcommand{\law}{\stackrel{\mathrm{d}}{=}}
\newcommand{\bmu}{\pmb{\mu}}
\newcommand{\bx}{\bm{x}}
\newcommand{\bX}{\bm{X}}

\newcommand{\bM}{\bm{M}}
\newcommand{\bU}{\bm{U}}
\newcommand{\bbm}{\bm{m}}
\newcommand{\bg}{\bm{g}}
\newcommand{\bW}{\bm{W}}
\newcommand{\bZ}{\bm{Z}}
\newcommand{\bmk}{\bbm_{k,n}}
\newcommand{\bMk}{\bM_{k,n}}
\newcommand{\bmuk}{\bmu_{k}}
\newcommand{\bWk}{\bW_{k}}

\newcommand{\bMkB}{\bM_{k,n}^*}
\newcommand{\bmukB}{\bmu_{k}^*}

\DeclareMathSymbol{\upSigma}{\mathalpha}{operators}{6}

\renewcommand{\bm}[1]{\textrm{$\pmb{#1}$}}

\newcommand{\SIG}{\mathrel{\mathpalette\SIS\relax}}
  \newcommand{\SIS}[2]{%
    \ooalign{$#1\pmb{\upSigma}$\cr\hfil$#1|$\hfil\cr}}
\newcommand{\Sig}{\operatorname{\SIG}}
\newcommand{\Sigk}{\Sig_k}
\newcommand{\BV}{\mathbf{V}}
\newcommand{\BH}{\mathbf{H}}
\newcommand{\BVk}{\BV_k}
 \newenvironment{pr}[1]{{\noindent\it{#1~}}}{\smallskip}

\makeatletter
\def\@seccntformat#1{\@ifundefined{#1@cntformat}%
   {\csname the#1\endcsname\quad}  % default
   {\csname #1@cntformat\endcsname}% enable individual control
}
\let\oldappendix\appendix %% save current definition of \appendix
\renewcommand\appendix{%
    \oldappendix
    \newcommand{\section@cntformat}{\appendixname~\thesection\quad}
}
\makeatother

\let\originalleft\left
\let\originalright\right
\renewcommand{\left}{\mathopen{}\mathclose\bgroup\originalleft}
\renewcommand{\right}{\aftergroup\egroup\originalright}

\usepackage{cleveref}
\crefname{section}{section}{sections}
\crefname{theorem}{theorem}{theorems}
\crefname{corollary}{corollary}{corollaries}
\crefname{lemma}{lemma}{lemmas}
\crefname{proposition}{proposition}{propositions}
\crefname{definition}{definition}{definitions}
\crefname{remark}{remark}{remarks}
\crefname{figure}{figure}{figures}
\crefname{equation}{equation}{equations}

\begin{document}

\title{On the limiting distribution of sample central moments}

%\titlerunning{Variance bounds for integrated Pearson family}        % if too long for running head

\author{Georgios Afendras \and Nickos Papadatos \and Violetta Piperigou}


\institute{Georgios Afendras
           \at
              Department of Mathematics, Aristotle University of Thessaloniki, 54124, Thessaloniki, Greece\\
              \email{gafendra@math.auth.gr}
           \and
           Nickos Papadatos
           \at
              Department of Mathematics, National and Kapodistrian University of Athens, Pane\-pi\-ste\-mi\-o\-po\-lis, 15784 Athens, Greece.\\
              \email{npapadat@math.uoa.gr}
           \and
           Violetta Piperigou
           \at
              University of Patras, 26500 Rion (Patras), Greece.\\
              \email{vpiperig@math.upatras.gr}
           }

\date{}%Received: date / Revised: date}

\maketitle

\begin{abstract}
 We investigate the limiting behavior of sample central moments, examining
 the special cases where the limiting (as the sample size tends to infinity)
 distribution is degenerate. Parent
 (non-degenerate) distributions with this property are called {\it singular},
 and we show in this article that the singular distributions contain at most
 three supporting points. Moreover, using the \emph{delta}-method,
 we show that the  (second order) limiting distribution of sample
 central moments from a singular distribution is either
 a multiple, or a difference of two multiples of independent chi-square
 random variables with one degree of freedom.
 Finally, we present a new characterization of normality through the
 asymptotic independence of the sample mean and all sample central moments.
\end{abstract}

\section{Introduction}
\label{sec.intr}

 Let $X$ be a random variable with distribution function $F$
 and finite moment of order $k$, for some positive integer $k\ge2$.
 Then, $X$ has finite central moment of order $k$.

 Based on a random sample of size $n$ from $F$,
 a natural estimator of the $k$th central moment of $X$ is
 the $k$th sample central moment, and
 the strong law of large numbers implies that
 the $k$th sample central moment is a strongly consistent estimator
 of the population $k$th central moment.
 If, in addition, $X$ has finite moment of order $2k$,
 its asymptotic normality is also known
 \citep[see, for example,][pp.~297--8]{Lehmann1999}.

 In the particular case where $k=2$ and the sample size $n\geq 2$
 is fixed, \citet{Kourouklis2012} proved that the usual
 unbiased estimator for the variance, $S^2$,  is inadmissible
 in the class $\CC=\{cS^2 \colon c>0\}$, showing that the estimator
 $n(n-1)[n(n-1)+2]^{-1}S^2$ has smaller mean squared error than $S^2$
 for all $F$ with finite fourth moment; see also \citet{Yatracos2005}.

 On the other hand, various authors provide statistical inference based on the
 asymptotic (as $n\to\infty$) distribution of a function of the sample
 central moments.
 Such results have several applications, including
 the evaluation of the limiting distribution of process capability indices,
 which have been widely used to measure
 the improvement of quality and productivity
 %; see, e.g., \citet{WuLiang2010}.
 \citep[see, e.g.,][]{WuLiang2010}.
 In a different context,
 \citet{Pewsey2005} and \citet{Afendras2013} provide hypothesis
 testing,
 including normality-testing, based on a function
 of the first four central moments of a distribution.

 \citet{HKLZ2007} suggest moment estimators for the parameters
 of a continuous time $\mathrm{GARCH}(1,1)$.
 The asymptotic normality of these estimators plays
 an important role in their analysis.

 Investigating the M-estimation procedure,
 \citet{SB2002} present cases in which central moment-based
 estimates may be presented as M-estimators.
 The asymptotic analysis and approximate inference
 are an important issue for large-sample inference.

 The sequence of the random vectors that contain the first $k$ central moments
 $\surd{n}$-converges in distribution to a $k$-dimensional normal distribution;
 this result arises easily from the multivariate central limit theorem
 and the \emph{delta}-method.
 However, there are cases where the asymptotic distribution of the
 $\surd{n}$-convergence of a central moment is a constant with probability one.
 In those cases, the order of convergence is faster than $\surd{n}$,
 specifically, the convergence is of order $n$.
 Therefore, a deeper study of the asymptotic behavior of
 the sample central moments is required.

 This paper is organized as follows.
 \Cref{sec.notation} provides the basic notation and terminology
 that will be used through the paper.
 \Cref{sec:motivation} presents a motivation
 of the problems that are studied and lists our contributions.
 \Cref{sec.lim_distr} provides the asymptotic
 distribution of the $\surd{n}$-convergence of the sample central moments,
 and discusses asymptotic properties of these moments.
 Specifically, we introduce the property
 of
 asymptotic independence,
 and investigate this property for the random vector
 of the first $k$ central moments;
 an asymptotic independence-based characterization for the normal
 distribution is also given.
 In \Cref{sec.singular_distr}
 we introduce the notion of a singular distribution
 and we study the class of such distributions,
 while \Cref{sec.Lim_distr_under_sing} contains
 results associated with the asymptotic distribution
 of sample central moments under singularity.
 Proofs of the results are presented
 in the Appendix.

\section{Notation and Terminology}
\label{sec.notation}

 Let $X\sim F$ with $\E|X|^{k}<\infty$ for some (fixed) $k\in\{1,2,\ldots\}$;
 and let us consider a random sample $X_1,\ldots,X_n$ from $F$.
 To avoid trivialities we further assume that $X$ is non-degenerate,
 that is, the set of points of increase of $F$,
 \[
 S_F\doteq\{x\in\R \colon F(x+\epsilon)-F(x-\epsilon)>0 \textrm{ for all }
 \epsilon>0\},
 \]
 contains at least two elements.
 The first $k$ central moments of $X$ around its mean, $\mu\doteq\E(X)$,
 are well-defined and finite. In the sequel, we shall use the notation
 \[
 \mu_j\doteq\E(X-\mu)^j, \quad j=0,\ldots,k.
 \]

 The sample moments of the centered $X_i$s are
 \[
 m_{j,n}\doteq\frac{1}{n}\sum_{i=1}^n (X_i-\mu)^j,
 \quad j=1,\ldots,k.
 \]

 The moment estimator of $\mu_k$ (for $k\geq 2$)
 when $\mu$ is unknown (as is usually the case)
 is its sample counterpart,
 \[
 M_{k,n}\doteq\frac{1}{n}\sum_{i=1}^n \left(X_i-\bar{X}_n\right)^k,
 \quad\textrm{where}\quad
 \bar{X}_n\doteq\frac{1}{n}\sum_{i=1}^n X_i;
 \]
 for convenience, we set $M_{1,n}\doteq\bar{X}_n-\mu$.

 Now, we define the vectors
 \[
 \bmuk\doteq(\mu_1,\ldots,\mu_k)'=\left(0,\sigma^2,\mu_3,\ldots,\mu_k\right)'
 \quad\textrm{and}\quad
 \bmukB\doteq\left(\sigma^2,\mu_3,\ldots,\mu_k\right)',
 \]
 as well as the random vectors
 \[
 \bMk\doteq(\bar{X}_n-\mu,M_{2,n},\ldots,M_{k,n})',
 \quad
 \bMkB\doteq(M_{2,n},\ldots,M_{k,n})',
 \quad
 \bmk\doteq(m_{1,n},\ldots,m_{k,n})';
 \]
 it is worth noting that it is convenient to find the asymptotic
 distribution
 of $\surd{n}(\bMk-\bmuk)$ instead of $\surd{n}(\bMkB-\bmukB)$.

 Observe that,
 %$m_{1,n}=M_{1,n}=\bar{X}_n-\mu$ and,
 by Newton's formula, $M_{j,n}=g_{j,k}(\bmk)$, where for $\bx_k=(x_1,\ldots,x_k)'$
 \[
 g_{j,k}(\bx_k)=(-1)^{j-1}(j-1)x_1^j +\sum_{i=2}^{j-1} (-1)^{j-i} {j\choose i}x_ix_1^{j-i}+x_j,
 \quad j=1,\ldots,k,
 \]
 where an empty sum should be treated as zero. Therefore,
 $\bMk=\bg_k(\bmk)$, where $\bg_k=(g_{1,k},\ldots,g_{k,k})'$.

 Finally, let $\bX_n$ be a sequence of random vectors.
 The terminology $\bX_n$ $\surd{n}$-converges in distribution
 to a distribution, say $F_{0}$, means that there exists
 $\bmu$ such that $\surd{n}(\bX_n-\bmu)\distr F_{0}$ as $n\to\infty$;
 similarly, we define $n$-convergence.
 In the rest of the paper, all limiting behaviors
 (limits, convergence in distribution or in probability as well as $o(\cdot)$, $O(\cdot)$ and $o_p(\cdot)$ functions)
 will be with respect to $n\to\infty$, except if something else is explicitly denoted.

\section{Motivation and our contributions}
\label{sec:motivation}

 Based on the asymptotic distribution
 of the vector of the sample skewness and kurtosis,
 \citet{Pewsey2005} gave an asymptotic
 result for testing normality.
 \citet{Afendras2013} establishes moment-based estimators
 of the parameter vector of the characteristic quadratic polynomial
 for both, integrated Pearson and cumulative Ord families of distributions,
 and obtained the asymptotic distribution of those estimators.
 Using this asymptotic distribution, he provides a number
 of hypothesis testing, including a normality test.
 In both cases, i.e., sample skewness and kurtosis \citep{Pewsey2005}
 and parameter vector of the characteristic quadratic polynomial \citep{Afendras2013},
 the estimator is a function of $\bM_{4,n}$.
 Thus,
 it is of some interest to obtain the asymptotic distribution of
 $\surd{n}(\bMk-\bmuk)$
 for any value of $k$.

 Our contributions are as follows.
 \begin{enumerate}[topsep=0ex, itemsep=.5ex, wide, labelwidth=!, labelindent=0pt, label=\rm\arabic*., ref=\textcolor{black}{\arabic*}]
 \item
 \label{contibutions1}
 We give some more light on the limiting behavior of the vector
 $\bMk$.
 In particular, we present results related
 to the rate of convergence of the first and second moments
 of $\bMk$.
 Furthermore, we investigate in some detail the singular cases, i.e.,
 the cases where $v_k^2=0$ (see \eqref{elements.V} below),
 characterizing the distributions with this property.

 \item
 We introduce the notion of {\it asymptotic independence}
 between the components of a sequence of $k$-dimensional random vectors,
 and we investigate the asymptotic properties of $\bMk$ in view of this notion.
 Specifically, we show that, among the distributions having finite moments of
 any order, the asymptotic independence of $\bar{X}_n$
 and the sequence $\{M_{k,n}, k\geq 2\}$ characterizes the
 normal distribution. This fact provides, in a sense, a
 limiting counterpart of the well-known result that independence of
 $\bar{X}_n$ and $M_{2,n}=(1-1/n)S^2_n$ (for some fixed $n\geq
 2$) characterizes normality
 %-- see \citet{Geary1936,Zinger1958,LLN1960,KLR1973}.
 \citep[see][]{Geary1936,Zinger1958,LLN1960,KLR1973}.
 Here, the assumption of independence is
 weakened to {\it asymptotic independence}
 but, of course, the requirement of the existence of all moments and
 the fact that $\bar{X}_n$ has to be asymptotically independent
 {\it of all} $M_{k,n}, k\geq 2$ (and not only $k=2$),
 seems to be quite restricted.
 However, this result is best possible.
 Indeed, as we shall show,
 %because
 for any fixed $k\geq 2$ there are (infinitely many)
 non-normal distributions for which $\bar{X}_n$
 and $\bMkB$ are asymptotically independent.

 \item
 Let $k=2,3,\ldots$ be fixed such that $\E|X|^{2k}<\infty$. Under non-singularity
 of order $k$, that is $v_k^2\ne0$, the $\surd{n}$-convergence
 of $M_{k,n}$ is a well-known result, i.e.,
 $\surd{n}(M_{k,n}-\mu_k)$ converges in distribution to $N(0,v_k^2)$.
 Under singularity of order $k$
 we shall verify the $n$-convergence of $M_{k,n}$, i.e.,
 $n(M_{k,n}-\mu_k)$ converges in distribution to a non-normal distribution.
 \end{enumerate}

\section{The limiting distribution and a characterization of normality}
\label{sec.lim_distr}

 Assume that $k\geq 2$ and $\E|X|^{2k}<\infty$.
 The multivariate central limit theorem immediately yields
 that
 \begin{equation}
 \label{clt}
 \surd{n}(\bmk-\bmuk)\distr
 N_k\left(\bm{0}_k,\Sigk\right),
 \end{equation}
 where $\bm{0}_k=(0,\ldots,0)'\in\R^k$
 and $\Sigk=(\sigma_{ij})\in\R^{k\times k}$ with
 $\sigma_{ij}=\mu_{i+j}-\mu_i\mu_j$.
 Since $\bMk=\bg_k(\bmk)$, the asymptotic distribution
 of $\surd{n}(\bMk-\bmuk)$ easily arises by a simple application of
 \emph{delta}-method and is a $k$-dimensional normal
 distribution \citep[see, e.g.,][Theorem 3.1]{vdVaart1998}.
 This result is presented in the following proposition.

 \begin{proposition}
 \label{prop.lim}
 If $\E|X|^{2k}<\infty$, then
 \begin{equation}
 \label{limit.gen}
 \surd{n}(\bMk-\bmuk)\distr N_k(\bm{0}_k,\BVk),
 \end{equation}
 where the variance-covariance matrix $\BVk=(v_{ij})\in\R^{k\times k}$
 has elements
 \begin{subequations}
 \label{elements.V}
 \begin{align}
 \label{elements.V1}
   v_{11} & = \sigma^2, &&  \\
   \label{elements.V2}
   v_{1j} & = v_{j1} = \mu_{j+1}-j\sigma^2\mu_{j-1}, && j=2,\ldots,k,\\
   \label{elements.V3}
   v_{ij} & = \mu_{i+j}-\mu_i\mu_j-i\mu_{i-1}\mu_{j+1}
 -j\mu_{i+1}\mu_{j-1}+ij\sigma^2\mu_{i-1}\mu_{j-1},&& i,j=2,\ldots,k;
 \end{align}
 \end{subequations}
 the elements $v_{ii}$ are also denoted by $v_i^2$, $i=1,\ldots,k$.
 \end{proposition}

 The proof of \Cref{prop.lim} for the case $k=4$
 is contained in \citet[in the proof of Theorem 3.1]{Afendras2013},
 while the proof for general $k$ is similar to the case $k=4$.
 Particular cases of the preceding result are contained in the next corollary.

 \begin{corollary}
 \label{cor.lim}
 If $k\geq 2$ and $\E|X|^{2k}<\infty$, then
 \begin{equation}
 \label{limit2}
 \surd{n}(M_{k,n}-\mu_k)\distr N\left(0,v_k^2\right);
 \end{equation}
 \begin{equation}
 \label{limit4}
 \surd{n}\left({\bar{X}_n-\mu\atop M_{k,n}-\mu_k}\right)\distr
 N_2\left(\left({0\atop0}\right),
 \left({\sigma^2 \atop \mu_{k+1}-k\sigma^2\mu_{k-1}}~~{\mu_{k+1}-k\sigma^2\mu_{k-1} \atop v_k^2}\right)\right).
 \end{equation}
 %\[
 %\surd{n}\left({\ds M_{r,n}-\mu_r \atop \ds M_{k,n}-\mu_k}\right)\distr N_2\left(\left({0\atop0}\right),
 %\left({v_r^2 \atop v_{rk}}~{v_{rk} \atop v_k^2}\right)\right)
 %\quad (2\leq r\leq k). \\
 %\]
 \end{corollary}

 Note that, as it is well-known,
 the weak convergence in \eqref{limit.gen}
 does not imply convergence of
 the corresponding moments; e.g., it is not
 necessarily true that either $\E[\surd{n}(M_{k,n}-\mu_k)]\to 0$
 or $\Var[ \surd{n}(M_{k,n}-\mu_k)]\to v_k^2$.
 Therefore, it is an interesting fact that
 \eqref{limit.gen} correctly suggests the limits for the corresponding
 expectations, variances and covariances.
 The following proposition
 %under some extra effort
 asserts that this moment convergence is indeed satisfied
 when the minimal (natural) set of assumptions is imposed on the
 moments of $X$; detailed proofs
 %of the following three Propositions
 are given in \Cref{appendix}.
 \begin{proposition}
 \label{prop.exp.cov.conv}
 Let $k,r\in\{2,3,\ldots\}$ be fixed.
 \begin{enumerate}[leftmargin=17pt, topsep=0ex, itemsep=.5ex, labelindent=0pt, label=\rm(\alph*), ref=\textcolor{black}{\alph*}]
 \item
 \label{prop.exp.cov.conv(a)}
 If $\E |X|^k<\infty$, then $\E(M_{k,n})=\mu_k+o(1/\surd{n})$;
 \item
 \label{prop.exp.cov.conv(b)}
 If $\E |X|^{k+1}<\infty$, then
 $\Cov(\bar{X}_n, M_{k,n})=(\mu_{k+1}-k\sigma^2\mu_{k-1})/n+o(1/n)$;
 \item
 \label{prop.exp.cov.conv(c)}
 If $\E |X|^{k+r}<\infty$,
 then
 $\Cov(M_{r,n}, M_{k,n})=v_{rk}/n+o(1/n)$,
 and in particular, if $\E|X|^{2k}<\infty$, then
 $\Var[\surd{n}(M_{k,n}-\mu_k)]\to v_{k}^2$.
 \end{enumerate}
 \end{proposition}

 In the sequel, we shall make use of
 the following definition.
 \begin{definition}
 \label{def.asympt.ind}
 Let
 %$\E|X|^{2k}<\infty$ for some
 $k\in\{2,3,\ldots\}$ be fixed.
 \begin{enumerate}[itemsep=.5ex, topsep=0ex, wide, labelwidth=!, labelindent=0pt, label=\rm(\alph*), ref=\textcolor{black}{\alph*}]
 \item
 \label{def.asympt.ind(a)}
 The sample mean, $\bar{X}_n$, is called
 {\it asymptotically independent}
 of the sample central moment, $M_{k,n}$, if
 there exist independent random variables $W_1$ and $W_k$ such
 that
 \[
 \surd{n}
 \left({\bar{X}_n-\mu\atop M_{k,n}-\mu_k}\right)
 \distr
 \left({W_1 \atop W_k}\right);
 \]
 \item
 \label{def.asympt.ind(b)}
 $\bar{X}_n$ is called {\it asymptotically independent}
 of the random vector $\bMkB$ if
 there exist random variables $W_1,\ldots,W_k$
 such that $W_1$, $\bW_{k}^*=(W_2,\ldots.W_k)'$ are independent
 and
 \[
 \surd{n}\left({\ds\bar{X}_n-\mu\atop\ds\bMkB-\bmukB}\right)
 \distr
 \left({\ds W_1\atop\ds \bW_{k}^*}\right);
 \]
 \item
 \label{def.asympt.ind(c)}
 $\bar{X}_n$ and $M_{k,n}$ are called
 {\it asymptotically uncorrelated} if
 \[
 \Cov\left(\surd{n}\bar{X}_n,\surd{n}M_{k,n}\right) \to 0.
 \]
 \end{enumerate}
 \end{definition}
 \begin{remark}
 Assume that $\E|X|^{2k}<\infty$ for some
 $k\in\{2,3,\ldots\}$. According to \eqref{limit4} and
 \Cref{prop.exp.cov.conv}\eqref{prop.exp.cov.conv(b)}
 (see, also, \eqref{cov.conv2}),
 $\bar{X}_n$ and $M_{k,n}$ are
 asymptotically independent if and only if they are
 asymptotically uncorrelated.
 Also, the asymptotic normality of \Cref{prop.lim}
 shows that $\bar{X}_n$ and $\bMkB$
 are asymptotically independent if and only if
 $\bar{X}_n$ and $M_{r,n}$ are
 asymptotically uncorrelated for all $r\in\{2,\ldots,k\}$.
 If we merely assume that $\E |X|^{k+1}<\infty$,
 then, even for those cases where
 the limiting distribution of $\surd{n}(M_{k,n}-\mu_k)$
 does not exist, \Cref{prop.exp.cov.conv}\eqref{prop.exp.cov.conv(b)}
 enables one to decide if  $\bar{X}_n$ and $M_{k,n}$ are
 asymptotically uncorrelated, or not.
 \end{remark}

 Assume now $X\sim N(\mu,\sigma^2)$.
 Observing the dispersion matrix in \eqref{limit.gen}, it becomes clear
 that the first column -- except of the first element, $\sigma^2$ -- vanish.
 This is so because $\mu_k=0$ for all odd $k$ and
 $\mu_{2r}=\sigma^{2r}(2r)!/(2^r r!)$; thus,
 for any  $k\in\{2,3,\ldots\}$, $\mu_{k+1}=k\sigma^2\mu_{k-1}$.
 According to \Cref{def.asympt.ind}, this means
 that $\bar{X}_n$ is asymptotically independent
 (uncorrelated) of all $M_{k,n}$.
 But, this is not a surprising fact for the normal
 distribution, since it is well-known that for any fixed
 $n\geq 2$, $\bar{X}_n$ is independent of the vector
 $\bZ\doteq(X_1-\bar{X}_n,\ldots,X_n-\bar{X}_n)'$
 (it suffices to observe that
 $(\bar{X}_n,X_1-\bar{X}_n,\ldots,X_{n}-\bar{X}_n)'$
 follows a multivariate normal distribution
 and $\Cov(\bar{X}_n,X_i-\bar{X}_n)=0$ for all $i$)
 and, therefore, $\bar{X}_n$ is stochastically independent
 (and uncorrelated) of any
 sequence of the form
 $\{h_r(\bZ), r=2,3,\ldots\}$, where $h_r\colon\R^n\to\R$
 are arbitrary Borel functions. Since
 $M_{r,n}={n}^{-1}\sum_{i=1}^n(X_i-\bar{X}_n)^r=h_r(\bZ)$,
 it follows that $\bar{X}_n$
 and $\bMkB$ are independent
 (and, thus, $\bar{X}_n$ and $M_{k,n}$ are uncorrelated)
 for all $k$ and $n$ and,
 certainly, the same is true for their limiting
 distributions. The interesting fact is that
 the converse is also true, i.e., the asymptotic independence
 of $\bar{X}_n$ and $M_{k,n}$ for all $k$ characterizes
 normality.
 \begin{theorem}
 \label{theo.char}
 Assume that $X$ is non-degenerate and has finite moments of any order.
 If $\bar{X}_n$ and $M_{k,n}$ are asymptotically
 independent {\rm (or, merely, asymptotically
 uncorrelated)}
 for all $k\in\{2,3,\ldots\}$,
 then $X$ follows a normal distribution.
 \end{theorem}
 \begin{proof}
 From \Cref{prop.exp.cov.conv}\eqref{prop.exp.cov.conv(b)}
 (cf.\ \eqref{limit.gen},
 \eqref{limit4}) it follows that $\bar{X}_n$ and $M_{k,n}$ are
 asymptotically uncorrelated if and only if
 $\mu_{k+1}=k\sigma^2\mu_{k-1}$. Since we have assumed that
 this relation holds {\it for all} $k\geq 2$ it follows
 that $\mu_1=\mu_3=\mu_5=\cdots=0$ and, similarly,
 for all $r\in\{1,2,\ldots\}$,
 $\mu_{2r}=\sigma^{2r}(2r)!/(2^r r!)$. But, these are the
 moments of $N(0,\sigma^2)$, and since normal distributions
 are characterized by their moment sequence
 %(see, e.g., Billingsley, 1995, Example 30.1, p.\ 389)
 \citep[see, e.g.,][Example 30.1, p.~389]{Billingsley1995},
 we conclude that
 $X-\mu\sim N(0,\sigma^2)$.
 \hfill\qed
 \end{proof}

 Compared to the classical characterization of normality
 via the independence of $\bar{X}_n$ and
 $S^2_n=[n/(n-1)]M_{2,n}$, the {\it asymptotic independence}
 is a much weaker condition to enable a characterization
 result. For example, \eqref{limit4} and
 \Cref{prop.exp.cov.conv}\eqref{prop.exp.cov.conv(b)}
 with $k=2$  (cf.\ \eqref{cov.conv2}) shows that
 $\bar{X}_n$ and $S^2_n$ are asymptotically independent if
 and only if $\mu_3=0$, provided $\E |X|^3<\infty$.
 Clearly, the relation $\E(X-\mu)^3=0$
 is satisfied by any symmetric distribution
 with finite third moment and by many others.
 On the other hand, the requirement that the asymptotic
 independence has to be fulfilled for all $k\geq 2$ may
 be regarded as too restricted.
 However, the following result shows that any finite
 number of $k$s will not work.

 \begin{theorem}
 \label{theo.legendre}
 For any fixed $k\geq 2$, there exist {\rm (infinitely many)}
 non-degenerate non-normal
 random variables $X$ with finite moments of any order such that
 $\bar{X}_n$ and $\bMkB$ are
 asymptotically independent.
 \end{theorem}
 \begin{proof}
 Let $\phi(x)\propto\exp(-x^2/2)$ be the
 standard normal density and consider the polynomial
 $P_{m}(x)\doteq(\ud{}^{m}/\ud{x}^{m})[x^{m}(1-x)^{m}]$;
 i.e., $P_m$ is the shifted Legendre polynomial of degree $m$.
 It is well-known that for all $m\geq k+2$,
 $P_{m}$ is orthogonal to $\{1,x,\ldots,x^{k+1}\}$ in the interval
 $[0,1]$, that is,
 \[
 \int_{0}^1 x^{j}P_{m}(x)\ud{x}=0, \quad j=0,\ldots,k+1.
 \]
 Since $P_m$ is continuous on $[0,1]$, it follows that
 $0<\max_{x\in[0,1]}|P_m(x)|\doteq a_m<\infty$. Also,
 $\min_{x\in[0,1]}\phi(x)=\phi(1)
 =(2\pi e)^{-1/2}>0$.
 Clearly, we can choose an $\epsilon_m>0$ small
 enough to guarantee that
 $\phi(x)+\epsilon_m P_m(x)>0$ for all $x\in[0,1]$
 (e.g., $\epsilon_m=[2a_m(2\pi e)^{1/2}]^{-1}$ suffices).
 Now, define a sequence of probability densities $\{f_m, \ m\geq k+2\}$
 by
 \[
 f_m(x)=\phi(x)+\epsilon_m P_m(x)\mathds{1}_{\{0\leq x\leq 1\}},
 \quad x\in\R,
 \]
 where $\mathds{1}$ denotes the indicator function (see \Cref{figure}).
 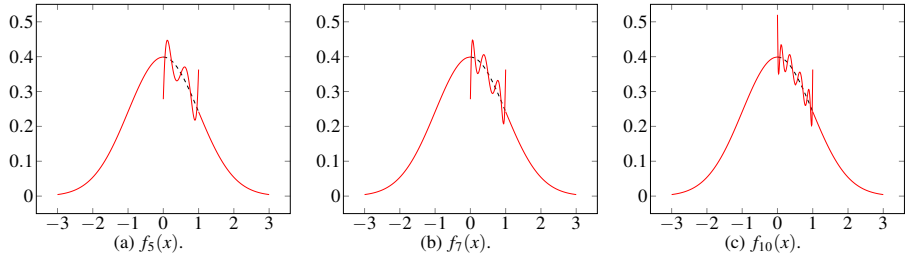
\begin{figure}[htp]
 \begin{subfigure}[a]{.32\textwidth}
 \centering
 \resizebox{\linewidth}{!}{
 \begin{tikzpicture}[declare function={PP(\x)=-(1 -30*\x +210*\x^2 -560*\x^3 +630*\x^4 -252*\x^5);
                                       P(\x)=PP(\x)*(\x<=0.5)-PP(1-\x)*(\x>0.5);}]
 \begin{axis}[ymax=.55, font=\Large,ytick={0,.1,.2,.3,.4,.5}]
 \addplot[red, domain=-3:0, samples=30,smooth]{exp(-x^2 /2)/(sqrt(2*pi))};
 \addplot[red, domain=1:3, samples=20,smooth]{exp(-x^2 /2)/(sqrt(2*pi))};
 \addplot[domain=0:1, samples=10,smooth,dashed]{exp(-x^2 /2)/(sqrt(2*pi))};
 \addplot[red, domain=0:1, samples=40,smooth]{exp(-x^2 /2)/(sqrt(2*pi))+P(x)/(2*4.133)};
 \end{axis}
 \end{tikzpicture}
 }
 \caption{$f_{5}(x)$.}
 \label{figure(a)}
 \end{subfigure}
 \hfill
 \begin{subfigure}[a]{.32\textwidth}
 \centering
 \resizebox{\linewidth}{!}{
 \begin{tikzpicture}[declare function={PP(\x)=-(1-56*\x+756*\x^2-4200*\x^3+11550*\x^4-16632*\x^5+12012*\x^6-3432*\x^7);
                                       P(\x)=PP(\x)*(\x<=0.5)-PP(1-\x)*(\x>0.5);}]
 \begin{axis}[ymax=.55, font=\Large,ytick={0,.1,.2,.3,.4,.5}]
 \addplot[red, domain=-3:0, samples=30,smooth]{exp(-x^2 /2)/(sqrt(2*pi))};
 \addplot[red, domain=1:3, samples=20,smooth]{exp(-x^2 /2)/(sqrt(2*pi))};
 \addplot[domain=0:1, samples=10,smooth,dashed]{exp(-x^2 /2)/(sqrt(2*pi))};
 \addplot[red, domain=0:1, samples=40,smooth]{exp(-x^2 /2)/(sqrt(2*pi))+P(x)/(2*4.133)};
 \end{axis}
 \end{tikzpicture}
 }
 \caption{$f_{7}(x)$.}
 \label{figure(b)}
 \end{subfigure}
 \hfill
 \begin{subfigure}[a]{.32\textwidth}
 \centering
 \resizebox{\linewidth}{!}{
 \begin{tikzpicture}[declare function={PP(\x)=1-110*\x+2970*\x^2-34320*\x^3+210210*\x^4-756756*\x^5+1681680*\x^6-2333760*\x^7+1969110*\x^8-923780*\x^9+184756*\x^(10);
                                       P(\x)=PP(\x)*(\x<=0.5)+PP(1-\x)*(\x>0.5);}]
 \begin{axis}[ymax=.55, font=\Large,ytick={0,.1,.2,.3,.4,.5}]
 \addplot[red, domain=-3:0, samples=30,smooth]{exp(-x^2 /2)/(sqrt(2*pi))};
 \addplot[red, domain=1:3, samples=20,smooth]{exp(-x^2 /2)/(sqrt(2*pi))};
 \addplot[domain=0:1, samples=10,smooth,dashed]{exp(-x^2 /2)/(sqrt(2*pi))};
 \addplot[red, domain=0:1, samples=52,smooth]{exp(-x^2/2)/(sqrt(2*pi)) +P(x)/(2*4.133)};
 \end{axis}
 \end{tikzpicture}
 }
 \caption{$f_{10}(x)$.}
 \label{figure(c)}
 \end{subfigure}
 %\includegraphics[width=.6\linewidth]{Fig}
 %\caption{The density $f_{21}(x)$.}
 \caption{The densities $f_{m}(x)$ for $m=5,7$ and $10$.}
 \label{figure}
 \end{figure}

 \noindent
 If $X_m$ has density $f_m$, it is clear that, for $j=0,\ldots, k+1$,
 \[
 \E\left(X_m^j\right) = \int_{\R} x^j \phi(x)\ud{x} +\epsilon_m \int_{0}^1
  x^j P_m(x) \ud{x} = \int_{\R} x^j \phi(x)\ud{x} =\E\left(Z^j\right),
 \]
 where $Z\sim N(0,1)$. Obviously,
 each  $X_m$ has finite
 moments of any order, is non-normal, non-degenerate,
 and, by \Cref{prop.lim},
 $\bar{X}_n$ and $\bMkB$
 are asymptotically independent.
 \hfill\qed
 \end{proof}

 \section{The singular distributions}
 \label{sec.singular_distr}

 First, center the rv $X$ as $U=X-\mu$
 with $\E(U^j)=\mu_j$ for all $j$.
 Assume that $\E|X|^{2k}<\infty$ for some $k=2,3,\ldots$
 and consider the random vector
 \[
 \bU_k=\left(U,U^2,U^3-3\sigma^2 U,U^4-4\mu_3 U,\ldots,U^k-k\mu_{k-1}U\right)'.
 \]
 It is of some interest to observe that
 the variance-covariance matrix of the limiting distribution
 in \eqref{limit.gen} coincides with the variance-covariance
 matrix of $\bU_k$. In particular,
 \begin{equation}
 \label{variance}
 v_k^2=\Var\left(U^k-k\mu_{k-1}U\right),
 \end{equation}
 \[
 \mu_{k+1}-k\sigma^2\mu_{k-1}
 =\Cov\left(U,U^k-k\mu_{k-1}U\right),
 \ \
 v_{rk}=\Cov\left(U^r-r\mu_{r-1}U,U^k-k\mu_{k-1}U\right),
 \]
 %It is of some interest to observe that for any
 %$r,k\in\{2,3,\ldots\}$,
 %\begin{equation}
 %\label{variance}
 %v_k^2=\Var\left[(X-\mu)^k-k\mu_{k-1} (X-\mu)\right],
 %\end{equation}
 %%\begin{equation}
 %%\label{covariance}
 %\[
 %v_{rk}=\Cov\left[(X-\mu)^r-r\mu_{r-1} (X-\mu),(X-\mu)^k-k\mu_{k-1}
 %(X-\mu)\right],
 %\]
 %%\end{equation}
 %and
 %%\begin{equation}
 %%\label{cov}
 %\[
 %\mu_{k+1}-k\sigma^2\mu_{k-1}
 %=
 %\Cov\left[X-\mu,(X-\mu)^k-k\mu_{k-1}(X-\mu)\right],
 %\]
 %%\end{equation}
 %provided that $\E|X|^{2k}<\infty$, $\E |X|^{r+k}<\infty$
 %and $\E |X|^{k+1}<\infty$, respectively. Therefore,
 %writing $Y=X-\mu$ with $\E(Y^k)=\mu_k$,
 %the dispersion matrix of the limiting distribution
 %in \eqref{limit.gen} coincides with
 %the dispersion matrix of the
 %random vector
 %\[
 %\bY_k=\left(Y,Y^2,Y^3-3\sigma^2 Y,Y^4-4\mu_3 Y,\ldots,Y^k-k\mu_{k-1}Y\right)'
 %\]
 $r=2,\ldots,k-1$.
 Relation \eqref{variance} evidently shows
 that $v_k^2\geq 0$. Of course,
 the non-negativity of $v_k^2$ is a consequence of
 the fact that, by \Cref{prop.exp.cov.conv}\eqref{prop.exp.cov.conv(c)},
 $v_k^2=\lim_n \Var(\surd{n}M_{k,n})$;
 but, the point here is that we have not to refer to a limit.
 Moreover, the expression \eqref{variance}
 enables to describe all distributions for which
 $v_k^2=0$. Such distributions will be called {\it
 singular}, according to the following definition.
 \begin{definition}
 \label{def.singular}
 For fixed $k\geq 2$, a non-degenerate random variable $X$, or its
 distribution function $F$, is
 called singular (of order $k$)
 if $\E|X|^{2k}<\infty$ and
 \[
 \surd{n}(M_{k,n}-\mu_k)\prob 0.
 \]
 The set of all singular random variables of order $k$ will
 be denoted by $\CF_k$; the subset of all standardized
 (with mean $0$ and variance $1$) singular random variables
 of order $k$ will be denoted by $\CF^0_k$.
 \end{definition}

 Noting that $Y\in \CF^0_k$ if and only if
 $X\doteq\mu+\sigma Y\in \CF_k$ for some $\mu\in\R$ and
 $\sigma>0$, it follows that $\CF_k$ contains exactly
 the location-scale family of the random variables
 that belong to $\CF_k^0$.
 According to \eqref{limit2}, \eqref{variance}, and
 \Cref{prop.exp.cov.conv}\eqref{prop.exp.cov.conv(a)},\eqref{prop.exp.cov.conv(c)}
 (cf.\ \eqref{exp.conv2}),
 $X\in\CF_k$ if and only if $v_k^2=0$ or, equivalently,
 \begin{equation}
 \label{singular}
 (X-\mu)^k=\mu_k+k\mu_{k-1}(X-\mu) \quad \textrm{with probability
 one}.
 \end{equation}
 We also note that $\CF_k$ is non-empty for all
 $k\geq 2$. Indeed, it is easily seen that
 the random variable $Y$ with $\Pr(Y=\pm 1)=1/2$
 belongs to $\CF_{2k}^0\subseteq \CF_{2k}$,
 $k=1,2,\ldots$, because $\mu=\E(Y)=0$, $\sigma^2=\E(Y^2)=1$,
 $\mu_{2k}=\E(Y^{2k})=1$, $\mu_{2k-1}=\E(Y^{2k-1})=0$ and
 $Y^{2k}=\mu_{2k}+2k\mu_{2k-1} Y$ with probability one.
 Similarly, for every $k\in\{1,2,\ldots\}$,
 the three-valued symmetric random variable $Y_{2k+1}$
 with $\Pr(Y_{2k+1}=\pm \sqrt{2k+1})=1/[2(2k+1)]$ and
 $\Pr(Y_{2k+1}=0)=2k/(2k+1)$ belongs to
 $\CF^0_{2k+1}$. Moreover, we shall show below
 (\Cref{lem.odd}) that we can find a unique
 value of $p=p_{2k+1}\in(1/2,1)$,
 for which the two-valued random variable $W_{2k+1}$,
 with
 \[
 \Pr\left(W_{2k+1}=\sqrt{{(1-p)}/{p}}\right)
 =p
 =1-\Pr\left(W_{2k+1}=-\sqrt{{p}/{(1-p)}}\right),
 \]
 is such that $W_{2k+1}\in \CF^0_{2k+1}$.

 In general, it is easily seen that the equation $y^k=\alpha+\beta y$
 (with $\alpha,\beta\in \R$) has at most two real solutions
 for even $k$, and at most three solutions for odd $k$.
 Assuming that $X\sim F$ and $X\in\CF_k$, it follows
 from \eqref{singular} that $X$ takes at most two values
 (and hence, exactly two values, since $X$ has been assumed to be
 non-degenerate) if $k$ is even, and two or three values if
 $k$ is odd. This follows from the fact
 that the points of increase of $F$ cannot be more
 than three, if $k$ is odd, and more than two, if $k$ is even.
 To see this assume, e.g., that $k$ is odd,
 $X\sim\CF_k$, $\E(X)=\mu$,
 $\E(X-\mu)^k=\mu_k$ and $\E(X-\mu)^{k-1}=\mu_{k-1}$.
 Let $x_1<x_2<x_3<x_4$ be four distinct points of increase of
 $F$. Then, there exists at least one $x_i$ for which
 $(x_i-\mu)^k-\mu_k-k\mu_{k-1}(x_i-\mu)\neq 0$, and thus, we can find
 a small $\epsilon>0$ such that $(x-\mu)^k\neq \mu_k+k\mu_{k-1}(x-\mu)$
 for
 all $x\in(x_i-\epsilon, x_i+\epsilon]$. Hence,
 $\Pr(x_i-\epsilon<X\leq x_i+\epsilon)\leq
 \Pr[(X-\mu)^k\neq
 \mu_k+k\mu_{k-1}(X-\mu)]$.
 % ; equivalently,
 % $(x-\mu)^k\neq \mu_k+k\mu_{k-1}(x-\mu)$ for
 % all $x\in(y_i-\mu-\epsilon, y_i+\mu-\epsilon]$.
 Since, however, $x_i$ is a point of increase of $F$,
 we have
 $0<F(x_i+\epsilon)-F(x_i-\epsilon)
 =\Pr(x_i-\epsilon<X\leq x_i+\epsilon)
 \leq\Pr[(X-\mu)^k \neq \mu_k+k\mu_{k-1}(X-\mu)]$,
 which contradicts \eqref{singular}.
 The same arguments apply to the case where $k$ is even.

 Therefore, we have the following description.
 \begin{proposition}
 \label{prop.singular}
 If $k\geq 2$ is even, then $\CF_k$ contains only
 two-valued random variables.
 If $k\geq 3$ is odd, then $\CF_k$ contains only
 two-valued and three-valued random variables.
 \end{proposition}

 Our purpose is to describe all singular distributions and
 to obtain a second order non-degenerate distributional limit
 for $M_{k,n}-\mu_k$.
 Firstly, we consider the two-valued distributions
 because they are possible members of $\CF_k$.
 \begin{lemma}
 \label{lem.even} Let $X\sim b(p)$, i.e.,
 $\Pr(X=1)=p=1-\Pr(X=0)$ for some $p\in(0,1)$.
 %, and assume that $k\geq 2$ is even.
 Then, $X\in\CF_2$ if and only if $p=1/2$.
 Moreover, if $k\geq 4$ is even, then
 %there exists a unique $p_k\in(1/2,1)$
 %such that
 $X\in \CF_k$
 if and only if $p\in\{1-p_k,1/2,p_k\}$,
 %The value of
 where $p_k$ is the unique root
 of the equation
 \begin{equation}
 \label{eq.even}
 \left(\frac{p}{1-p}\right)^{k-1}=\frac{(k+1)p-1}{k-(k+1)p},
 \qquad \frac{k-2}{k-1}<p<\frac{k}{k+1};
 \end{equation}
 in particular, $p_4={1}/{2}+{\surd{3}}/{6}$
 and $p_6={1}/{2}+{\sqrt{15(4\surd{10}-5)}}/{30}$.
 \end{lemma}
 \begin{proof}
 Since $\mu=p$ and
 \begin{equation}
 \label{mu_k}
 \mu_k=p(1-p)\left[(1-p)^{k-1}+(-1)^k p^{k-1}\right],
 \quad k=1,2,\ldots,
 \end{equation}
 \eqref{singular} shows that $X\in \CF_k$ if and only
 if
 \begin{equation}
 \label{eq.sing}
 (x-p)^k=\mu_k+k\mu_{k-1}(x-p) \quad\textrm{for } x=0  \textrm{ and } x=1.
 \end{equation}
 Using \eqref{mu_k} and the fact that $k\geq 2$ is even,
 both equations in \eqref{eq.sing} are reduced to
 \begin{equation}
 \label{eq.even2}
 p^{k-1} [k-(k+1)p] = (1-p)^{k-1} [(k+1)p-1], \quad 0<p<1.
 \end{equation}
 Since the value of $p=k/(k+1)$ is a root of the lhs of
 \eqref{eq.even2} which is not a root of its rhs we conclude that
 \eqref{eq.even} and \eqref{eq.even2} are equivalent.
 Obviously, $p=1/2$ is a root of \eqref{eq.even2}.
 In order to find all roots of \eqref{eq.even2},
 we make the substitution $t={p}/{(1-p)}$,
 which monotonically maps $p\in(0,1)$
 to $t\in(0,\infty)$.
 Then, we get the equation
 \begin{equation}
 \label{eq.even3}
 p_k(t)\doteq t^k-k t^{k-1}+k t-1=0, \quad 0<t<\infty,
 \end{equation}
 which has the obvious solution $t=1$ (corresponding to $p=1/2$).
 If $k=2$, then \eqref{eq.even3} is written as $t^2-1=0$, and
 thus $t=1$ (resp.\ $p=1/2$) is the unique solution
 of \eqref{eq.even3} (resp.\ \eqref{eq.even2}).
 Since for any $t>0$ we have $p_k(1/t)=-p_k(t)/t^k$, it
 follows that $1/t$ is a root of \eqref{eq.even3} whenever
 $t$ is; equivalently, $1-p$ is a root of \eqref{eq.even2}
 if $p$ is. For even $k\geq 4$ we see that $p_k(0)=-1$,
 $p_k(1)=0$ and $p_k(\infty)\doteq \lim_{t\to\infty}p_k(t)=\infty$. Also,
 $p_k'(t)=k[t^{k-1}-(k-1)t^{k-2}+1]=k q_k(t)$, where
 $q_k(t)=t^{k-1}-(k-1)t^{k-2}+1$  satisfies $q_k(0)=1>0$,
 $q_k(1)=-(k-3)<0$ and $q_k(\infty)=\infty$. Moreover, we
 see that $q_k'(t)=(k-1)t^{k-3}[t-(k-2)]$ is negative for
 $t\in(0,k-2)$ and positive for $t\in(k-2,\infty)$; thus,
 $q_k(t)$ decreases in $(0,k-2)$ and increases in
 $(k-2,\infty)$.
 Therefore, there exist $\rho_1<\rho_2$, with $0<\rho_1<1<k-2<\rho_2<\infty$,
 such that $q_k(t)<0$ for $t$ in $(0,\rho_1)\cup (\rho_2,\infty)$
 and $q_k(t)>0$ for $t$ in $(\rho_1,\rho_2)$.
 Relation $q_k(t)=p_k'(t)/k$ shows that $p_k(t)$
 is increasing in $(0,\rho_1)$, decreasing in $(\rho_1,\rho_2)$ and
 increasing in $(\rho_2,\infty)$. From
 $1\in(\rho_1,\rho_2)$ and $p_k(1)=0$, we conclude that
 $p_k(\rho_1)>0$ and $p_k(\rho_2)<0$; hence, there exist
 unique $t_1\in(0,\rho_1)$ and $t_2\in(\rho_2,\infty)$
 such that $p_k(t_1)=0=p_k(t_2)$.
 Clearly, $t_1=1/t_2$, and the set of roots of
 \eqref{eq.even3} is $\{1/t_2,1,t_2\}$; thus,
 the set of roots of \eqref{eq.even2} is $\{1-p_k,1/2,p_k\}$,
 with $p_k={t_2}/{(1+t_2)}$. Finally, relation $t_2>\rho_2>k-2$
 shows that
 $p_k={t_2}/{(1+t_2)}>{(k-2)}/{[1+(k-2)]}={(k-2)}/{(k-1)}$, while
 $p_k<{k}/{(k+1)}$ is obvious because for $p\geq {k}/{(k+1)}$
 the lhs of \eqref{eq.even2} is
 non-positive while its rhs is strictly positive.
 \hfill\qed
 \end{proof}

 From \eqref{eq.even}, we see that
 ${1}/{2}<p_4<p_6<\cdots$ and $p_{2k}=1-{1}/{(2k)}+o(1/k)$
 as $k\to\infty$. \Cref{lem.even} completely describes
 all $\CF_k$ for even $k$.
 \begin{corollary}
 \label{cor.even} If $k\geq 2$ is even, then
 $X\in\CF_k^0$ if and only if either
 $\Pr(X=\pm 1)={1}/{2}$,  or
 \[
 \Pr\left(X=-\sqrt{{p_k}/{(1-p_k)}}\right)=1-p_k,
 \qquad
 \Pr\left(X=\sqrt{{(1-p_k)}/{p_k}}\right)=p_k
 \]
 and $k\in\{4,6,\ldots\}$, or
 \[
 \Pr\left(X=-\sqrt{{(1-p_k)}/{p_k}}\right)=p_k,
 \qquad
 \Pr\left(X=\sqrt{{p_k}/{(1-p_k)}}\right)=1-p_k
 \]
 and $k\in\{4,6,\ldots\}$, where $p_k\in((k-2)/(k-1),k/(k+1))$ is given by \eqref{eq.even}.
 \end{corollary}
 \Cref{cor.even} says that $\CF_2^0$
 is singleton and that for every $k\in\{4,6,\ldots\}$,
 $\CF_k^0$ contains exactly three
 two-valued distributions.
 When $k$ is odd the nature of $\CF_k$
 is quite different, because it contains
 both two-valued and three-valued
 distributions.
 First we examine the two-valued case.

 \begin{lemma}
 \label{lem.odd} Let $X\sim b(p)$ for some
 $p\in(0,1)$.
 %, and assume that $k\geq 2$ is even.
 %Then $X\in\CF_2$ if and only if $p=1/2$.
 If $k\geq 3$ is odd
 and
 $X\in\CF_k$,
 then
 $p\in\{1-p_k,p_k\}$
 where
 $p_k$ is the unique root
 of the equation
 \begin{equation}
 \label{eq.odd}
 \left(\frac{p}{1-p}\right)^{k-1}=\frac{(k+1)p-1}{(k+1)p-k},
 \qquad
 \frac{k}{k+1}<p<1;
 \end{equation}
 in particular, $p_3={1}/{2}+{\surd{3}}/{6}$
 and $p_5={1}/{2}+{\sqrt{5\surd{5}}}/{10}$.
 Conversely, if $k\geq 3$ is odd and either $X\sim b(p_k)$ or $X\sim
 b(1-p_k)$,
 with $p_k$ as above,
 then $X\in\CF_k$.
 \end{lemma}
 \begin{proof}
 Assume that $k\geq 3$ is odd, $X\sim b(p)$ and
 $X\in\CF_k$. This means that
 %\eqref{singular}, i.e.,
 \eqref{eq.sing}
 is satisfied.
 Using \eqref{mu_k} and the fact that $k\geq 3$ is odd,
 both equations in \eqref{eq.sing} are reduced to
 \begin{equation}
 \label{eq.odd2}
 p^{k-1} [(k+1)p-k] = (1-p)^{k-1} [(k+1)p-1], \quad 0<p<1.
 \end{equation}
 Since the value of $p=k/(k+1)$ is a root of the lhs of
 \eqref{eq.odd2} which is not a root of its rhs, we conclude that
 \eqref{eq.odd} and \eqref{eq.odd2} are equivalent.
 As in \Cref{lem.even}, in order to find all
 roots of \eqref{eq.odd2} we make the substitution
 $t={p}/{(1-p)}$, which monotonically maps $p\in(0,1)$
 to $t\in(0,\infty)$.
 Then, we get the equation
 \begin{equation}
 \label{eq.odd3}
 p_k(t)\doteq t^k-k t^{k-1}-k t+1=0, \quad 0<t<\infty.
 \end{equation}
 Since for any $t>0$ we have $p_k(1/t)=p_k(t)/t^k$, it
 follows that $1/t$ is a root of \eqref{eq.odd3} whenever
 $t$ is; equivalently, $1-p$ is a root of \eqref{eq.odd2}
 if $p$ is. For odd $k\geq 3$, we see that $p_k(0)=1>0$,
 $p_k(1)=-2(k-1)<0$ and $p_k(\infty)=\infty$. Thus,
 \eqref{eq.odd3} has at least one root in $(0,1)$
 and at least one root in $(1,\infty)$.
 Also, $p_k'(t)=k[t^{k-1}-(k-1)t^{k-2}-1]=k q_k(t)$, where
 $q_k(t)=t^{k-1}-(k-1)t^{k-2}-1$  satisfies $q_k(0)=-1<0$
 %$q_k(1)=-(3-k)<0$
 and $q_k(\infty)=\infty$.
 Moreover, we see that $q_k'(t)=(k-1)t^{k-3}[t-(k-2)]$ is
 negative for $t\in(0,k-2)$ and positive for
 $t\in(k-2,\infty)$; thus,
 $q_k(t)$ decreases in $(0,k-2)$ and increases in
 $(k-2,\infty)$.
 Therefore, there exists a unique $\rho>k-2\geq 1$
 such that $q_k(t)<0$ for $t$ in $(0,\rho)$
 and $q_k(t)>0$ for $t$ in $(\rho,\infty)$.
 Relation $q_k(t)=p_k'(t)/k$ shows that $p_k(t)$
 decreases in $(0,\rho)$ and increases in $(\rho,\infty)$.
 From $p_k(0)>0$, $p_k(1)<0$ and $p_k(\infty)>0$
 we conclude that there exist unique $t_1\in(0,1)$ and
 $t_2\in(\rho,\infty)$ such that $p_k(t_1)=0=p_k(t_2)$.
 Clearly, $t_1=1/t_2$, and the set of roots of
 \eqref{eq.odd3} is $\{1/t_2,t_2\}$; thus,
 the set of roots of \eqref{eq.odd2} is $\{1-p_k,p_k\}$,
 with $p_k={t_2}/{(1+t_2)}$. Finally, relation
 $t_2>\rho>k-2$ shows that
 $p_k={t_2}/{(1+t_2)}>{(k-2)}/{[1+(k-2)]}={(k-2)}/{(k-1)}$.
 However, the root $p_k$ cannot lie in $({(k-2)}/{(k-1)},{k}/{(k+1)}]$
 because for all $p$ in this interval the lhs of \eqref{eq.odd2}
 is non-positive, while its rhs is strictly positive
 ($p>{(k-2)}/{(k-1)}$ implies $(k+1)p-1>(k+1){(k-2)}/{(k-1)}-1
 ={[(k-3)(k+1)+2]}/{(k-1)}>0$,
 since $k\geq 3$). This verifies that $p_k>{k}/{(k+1)}$.
 Finally, if either $X\sim b(p_k)$ or $X\sim b(1-p_k)$, then
 \eqref{mu_k} and \eqref{eq.odd2} show that \eqref{eq.sing}
 and \eqref{singular} are satisfied and, thus, $X\in\CF_k$.
 \hfill\qed
 \end{proof}
 \begin{corollary}
 \label{cor.odd} If $k\geq 3$ is odd, then the
 unique two-valued random variables contained in
 $\CF_k^0$ are the following:
 \[
 \Pr\left(X=-\sqrt{{p_k}/{(1-p_k)}}\right)=1-p_k,
 \qquad
 \Pr\left(X=\sqrt{{(1-p_k)}/{p_k}}\right)=p_k
 \]
 and
 \[
 \Pr\left(X=-\sqrt{{(1-p_k)}/{p_k}}\right)=p_k,
 \qquad
 \Pr\left(X=\sqrt{{p_k}/{(1-p_k)}}\right)=1-p_k,
 \]
 where $p_k\in({k}/{(k+1)},1)$ is given by \eqref{eq.odd}.
 \end{corollary}
 \Cref{cor.odd} describes all two-valued random
 variables of $\CF_k^0$ when $k\geq 3$ is odd;
 however, we have already seen that
 $\CF_k^0$ contains also some
 three-valued random variables.
 Among them, exactly one is symmetric.
 \begin{lemma}
 \label{lem.sym}
 If $k\geq 3$ is odd, the unique symmetric random
 variable of $\CF_k^0$ is given by
 \[
 \Pr(X=\pm \surd{k})=1/(2k),
 \qquad
 \Pr(X=0)=1-1/k.
 \]
 More generally, this is the unique random variable of
 $\CF_k^0$ with $\mu_k=0$.
 \end{lemma}
 \begin{proof}
 Since $X\in\CF_k^0$, we have $\E(X)=0$ and $\E(X^2)=1$.
 %Since $X$ is symmetric and $k$ is odd we have $\mu_k=0$.
 Therefore, in view of the assumption $\mu_k=0$,
 \eqref{singular} simplifies to
 \[
 X \left(X^{k-1}-k\mu_{k-1}\right)=0 \quad \textrm{with probability one.}
 \]
 It follows that the support of $X$ is a subset
 of $A\doteq\{-(k\mu_{k-1})^{{1}/{(k-1)}},0,(k\mu_{k-1})^{{1}/{(k-1)}}\}$,
 where $\mu_{k-1}=\E(X^{k-1})>0$, because $X$ is non-degenerate and
 $k-1$ is even. Let $p=\Pr(X=0)$, $p_1=\Pr(X=-(k\mu_{k-1})^{{1}/{(k-1)}})$
 and $p_2=\Pr(X=(k\mu_{k-1})^{{1}/{(k-1)}})$; $p$, $p_1$, $p_2$
 are non-negative and $p+p_1+p_2=1$ because
 $\Pr(X\in A)=1$. Assumption $\E(X)=0$ shows that $p_1=p_2$. Thus,
 $p_1=p_2={(1-p)}/{2}$ and, so,
 $\Pr(X=\pm (k\mu_{k-1})^{{1}/{(k-1)}})={(1-p)}/{2}$.
 Calculating $\mu_{k-1}=\E(X^{k-1})=(1-p)k\mu_{k-1}$, we see that
 $p=1-{1}/{k}$ and thus, $\Pr(X=\pm a)={1}/{(2k)}$
 where $a=(k\mu_{k-1})^{{1}/{(k-1)}}>0$.
 Finally, from $1=\E(X^2)={a^2}/{k}$, we conclude that
 $a=\surd{k}$. On the other hand, it is easily seen that
 for this value of $a=\surd{k}$, $\mu_k=0$ and
 $\mu_{k-1}=k^{(k-3)/2}$ so that
 $k\mu_{k-1}=k^{(k-1)/2}=(\pm\surd{k})^{k-1}$; hence,
 $A=\{-\surd{k},0,\surd{k}\}$  and
 $x(x^{k-1}-k\mu_{k-1})=x[x^{k-1}-(\pm\surd{k})^{k-1}]\equiv 0$
 for all $x\in A$.
 \hfill\qed
 \end{proof}

 The following lemma presents a complete description
 of all tree-valued distributions of $\CF_3^0$
 and gives a picture of the nature of
 $\CF_k^0$ when $k\geq 3$ is odd.
 \begin{lemma}
 \label{lem.F_3}
 For each $\mu_3\in[-\surd{2},\surd{2}]$
 there exists a unique random variable $X\in \CF_3^0$
 such that $\E(X^3)=\mu_3$.
 Cases $\mu_3=\pm\surd{2}$ correspond
 to the two-valued distributions described in \Cref{cor.odd}
 for $k=3$. Any other value
 of $\mu_3\in(-\surd{2},\surd{2})$ uniquely determines
 a three-valued distribution, and in particular,
 $\mu_3=0$ corresponds to the symmetric distribution of
 \Cref{lem.sym} for $k=3$. Moreover, there not exist
 other random variables in $\CF_3^0$.
 Therefore, $\CF_3^0$ admits the parametrization
 %\begin{equation}
 %\label{F_3}
 \[
 \CF_3^0=\{X_\theta, \ -\surd{2}\leq
 \theta\leq \surd{2}\},
 \]
 %\end{equation}
 where $X_{\theta}$ is characterized by
 %\begin{equation}
 %\label{F_3_char}
 \[
 \E\left(X_{\theta}\right)=0, \quad \E\left(X_{\theta}^2\right)=1, \quad \E\left(X_\theta^3\right)=\theta
 \quad \textrm{and} \quad \Pr\left[X_{\theta}\left(X_{\theta}^2-3\right)=\theta\right]=1.
 \]
 %\end{equation}
 \end{lemma}
 \begin{proof}
 Let $X\in\CF_3^0$ and assume that $\mu_3=\E(X^3)=\theta$.
 Then, $\mu=\E(X)=0$, $\sigma^2=\mu_2=\E(X^2)=1$ and, according to \eqref{singular},
 $X(X^2-3)=\theta$ with probability one. Therefore, since $X$ is non-degenerate,
 the support of $X$ must contains at least two points which
 are included in the set of zeros of $y(y^2-3)=\theta$.
 This shows that $|\theta|\leq 2$ because, otherwise, the
 set $\{y\in\R\colon y(y^2-3)=\theta\}$ is a singleton.
 Observe that $\theta=0$ leads, uniquely,
 to the symmetric random variable of \Cref{lem.sym}
 with $k=3$. Thus, from now on assume that $\theta\neq 0$.
 The values $\theta=\pm 2$ are impossible because
 the equations $y(y^2-3)=\pm 2$ have exactly two real
 solutions, say $\alpha,\beta$, with $|\alpha|=1$  and
 $|\beta|=2$, so that $\E(X)=0$ and $\E(X^2)=1$ are impossible.

 Consider now the case where $-2<\theta<0$. Then,
 $\{y\colon y(y^2-3)=\theta\}=\{-\alpha,\beta,\gamma\}$
 where $0<\beta<1<\gamma<\surd{3}<\alpha$
 and, by definition, the numbers
 $\alpha$, $\beta$, $\gamma$ satisfy the relation
 \begin{equation}
 \label{abc}
 -\alpha\left(\alpha^2-3\right) = \beta\left(\beta^2-3\right)=\gamma\left(\gamma^2-3\right)
 =\theta.
 \end{equation}
 From \eqref{abc}, we see that
 $\theta=\theta(\beta)=\beta(\beta^2-3)$ is a strictly
 decreasing and continuous function of $\beta$ which maps
 $\beta\in(0,1)$ to $\theta\in(-2,0)$; thus, its inverse function,
 $\beta(\theta)\colon(-2,0)\to(0,1)$,
 is well-defined, continuous and strictly
 decreasing in $\theta$ with
 $\beta(-2_+)=1$ and $\beta(0_-)=0$.
 Also, from \eqref{abc} we get the equation
 $3(\alpha+\beta)=\alpha^3+\beta^3
 =(\alpha+\beta)(\alpha^2-\alpha\beta+\beta^2)$
 which shows that $\alpha^2-\alpha\beta+\beta^2=3$ and,
 since $\alpha>{\beta}/{2}$, we have
 \begin{equation}
 \label{alpha}
 \alpha=\alpha(\beta)=\frac{1}{2}(\beta+\delta),
 \quad
 \textrm{where}
 \
 \delta=\delta(\beta)=\sqrt{3\left(4-\beta^2\right)}.
 \end{equation}
 Similarly,
 \eqref{abc} yields the equation
 $3(\gamma-\beta)=\gamma^3-\beta^3
 =(\gamma-\beta)(\gamma^2+\gamma\beta+\beta^2)$
 which shows, in view of $\beta<\gamma$,
 that $\gamma^2+\gamma\beta+\beta^2=3$.
 Since $\gamma>0$ it follows that
 \begin{equation}
 \label{gamma}
 \gamma=\gamma(\beta)=\frac{1}{2}(-\beta+\delta),
 \quad
 \textrm{where}
 \
 \delta=\delta(\beta)=\sqrt{3\left(4-\beta^2\right)}.
 \end{equation}
 From \eqref{alpha} and \eqref{gamma} we conclude that
 $\alpha=\beta+\gamma$. Set now $p_1=\Pr(X=-\alpha)$,
 $p_2=\Pr(X=\beta)$ and $p_3=\Pr(X=\gamma)$.
 Since $\Pr(X\in\{-\alpha,\beta,\gamma\})=1$
 and $\E(X)=0$, $\E(X^2)=1$, we get the system of
 equations (in $p_1$, $p_2$, $p_3$)
 \[
 p_1+ p_2 + p_3=1,
 \quad
 -\alpha p_1+\beta p_2 +\gamma p_3 =0,
 \quad
 \alpha^2 p_1+\beta^2 p_2 +\gamma^2 p_3 = 1,
 \]
 which, in view of $\alpha=\beta+\gamma$,
 has the unique solution
 \[
 p_1=\frac{1+\beta\gamma}{(\beta+2\gamma)(2\beta+\gamma)},
 \quad
 p_2=\frac{\gamma(\beta+\gamma)-1}{(\gamma-\beta)(2\beta+\gamma)},
 \quad
 p_3=\frac{1-\beta(\beta+\gamma)}{(\gamma-\beta)(\beta+2\gamma)}.
 \]
 Now, since $\gamma^2+\gamma\beta+\beta^2=3$, we have $\gamma(\beta+\gamma)=3-\beta^2$
 and $\beta(\beta+\gamma)=3-\gamma^2$; substituting these
 values in the numerators of $p_2$ and $p_3$ we get
 \begin{equation}
 \label{solution}
 p_1=\frac{1+\beta\gamma}{(\beta+2\gamma)(2\beta+\gamma)},
 \quad
 p_2=\frac{2-\beta^2}{(\gamma-\beta)(2\beta+\gamma)},
 \quad
 p_3=\frac{\gamma^2-2}{(\gamma-\beta)(\beta+2\gamma)}.
 \end{equation}
 It is clear that $p_1>0$ and $p_2>0$ for all values of
 $\beta$ and $\gamma$ with (see \eqref{gamma})
 \[
 0<\beta<1<\gamma=\frac{-\beta+\sqrt{3\left(4-\beta^2\right)}}{2}<\surd{3}.
 \]
 However, this is not the case for $p_3$, since $p_3\geq 0$
 requires $\gamma^2\geq 2$, i.e., $\gamma\geq \surd{2}$
 (since $\gamma>0$).
 Now, from $\mu_3=\theta=\gamma(\gamma^2-3)$ and the fact
 that $\gamma\in[\surd{2},\surd{3})$, we conclude that all
 possible values of $\theta$ (with $\theta<0$) are
 included in the interval $[-\surd{2},0)$.
 Using \eqref{gamma} and
 the fact that $\beta>0$, it follows that
 $\gamma\geq \surd{2}$ if and only if
 $0<\beta\leq {(\surd{6}-\surd{2})}/{2}=\sqrt{2-\surd{3}}$.
 Now, observe that $\gamma=\surd{2}$ corresponds to
 a standardized two-valued random variable with
 $\mu_3=\theta=\gamma(\gamma^2-3)=-\surd{2}$,
 taking the values $-\alpha=-\beta-\gamma=-{(\surd{6}+\surd{2})}/{2}=-\sqrt{2+\surd{3}}$
 and $\beta=\sqrt{2-\surd{3}}={(\surd{6}-\surd{2})}/{2}$ with respective
 probabilities $1-p$ and $p$,
 where
 \[
 p=\frac{2-\beta^2}{(\gamma-\beta)(2\beta+\gamma)}
 =\frac{1}{2}+\frac{\surd{3}}{6};
 \]
 this is the first two-valued random variable of \Cref{cor.odd}
 when $k=3$. On the other hand, each value of
 $\gamma\in(\surd{2},\surd{3})$ corresponds to
 a unique value of
 $\mu_3=\theta=\gamma(\gamma^2-3)\in(-\surd{2},0)$,
 which, in turn, uniquely determines
 $\beta=\beta(\theta)\in(0,{(\surd{6}-\surd{2})}/{2})$
 through $\beta=[-\gamma+\sqrt{3(4-\gamma^2)}]/2$
 (cf.\ \eqref{gamma}),
 and $\alpha=\alpha(\theta)$ through $\alpha=\beta+\gamma$.
 Finally, these uniquely determined values of $\alpha$, $\beta$
 and $\gamma$ specify the (strictly positive)
 probabilities $p_1$, $p_2$ and $p_3$, through
 \eqref{solution}, which shows that each $X_{\theta}\in\CF_{3}^0$ is
 uniquely determined by $\E(X_{\theta}^3)=\theta$,
 $-\surd{2}<\theta<0$.

 It remains to examine the cases where $0<\theta<2$.
 However, if $X\in \CF_{3}^0$ and $\E(X^3)=\theta>0$,
 then it is easily verified that $-X\in \CF_{3}^0$ and
 $\E(-X)^3=-\theta<0$. By the previous arguments it follows
 that, necessarily, $-\surd{2}\leq -\theta<0$, that $-X$ is determined
 by the value of $-\theta$, and that $-X$ is a two-valued
 random variable, if $-\theta=-\surd{2}$, and a three-valued
 random variable otherwise; thus, the same is true for $X$,
 and the proof is complete.
 \hfill\qed
 \end{proof}

 We was not able to completely describe
 $\CF_k^0$ for odd $k\geq 5$. However,
 the situation seems to be similar
 to the case $k=3$, i.e., each $X_{\theta}\in\CF_k^0$
 is characterized by its central moment,
 $\theta=\E(X^k)=\mu_k$, and the possible values of $\theta$
 form a symmetric interval of the form
 $[-\alpha_k,\alpha_k]$, where the boundary values $\theta=\pm \alpha_k$
 correspond to the two-valued distributions of
 \Cref{cor.odd}, while every $\theta\in(-\alpha_k,\alpha_k)$
 determines uniquely a three-valued random variable.

 \section{Limiting distribution under singularness}
 \label{sec.Lim_distr_under_sing}

 If the random sample comes from a singular distribution of order $k\geq 2$,
 then the asymptotic normality of \eqref{limit2} reduces to
 $\surd{n}(M_{k,n}-\mu_k)\prob 0$
 (see \Cref{def.singular}).
 This shows that the order of
 convergence of $M_{k,n}$ to $\mu_k$
 is faster than $o(1/\surd{n})$, and
 a second order approximation applies, according
 to the following lemma.
 \begin{lemma}
 \label{lem.delta2}
 Assume that $\bX_n$ is a sequence of $k$-variate random vectors
 such that
 \begin{equation}
 \label{lim}
 \surd{n}(\bX_n-\bmu)\distr \bW,
 \end{equation}
 where $\bmu\in\R^k$ and $\bW$ is a
 $k$-variate random vector.
 % with mean vector
 % $\E\left(\bW\right)=\bm{0}$ and dispersion matrix
 % $\D(\bW)=\Sig$.
 Suppose that the Borel function
 $g\colon\R^k\to\R$ is twice continuously differentiable
 at a neighborhood of $\bmu$
 and define
 \[
 \nabla g(\bmu) =
 \left.\left(\frac{\partial g(\bx)}{\partial
 x_i}\right)\right|_{\bx=\bmu}\in \R^{k}
 \quad \textrm{and} \quad
 \BH_k(\bmu)\doteq
 \left.\left(\frac{\partial^2
 g(\bx)}{\partial x_i \partial x_j}\right)\right|_{\bx=\bmu}\in\R^{k\times
 k}.
 \]
 If
 \begin{equation}
 \label{sing.var}
 %[\nabla g(\bmu)]'\Sig [\nabla g(\bmu)]=0
 n [\nabla g(\bmu)]' (\bX_n-\bmu)\prob 0,
 \end{equation}
 then
 \begin{equation}
 \label{delta2}
 n[g(\bX_n)-g(\bmu)]\distr
 \frac{1}{2} \bW'\BH_k(\bmu)\bW.
 \end{equation}
 \end{lemma}
 \begin{proof}
 By \eqref{lim}, we see that $\bX_n\prob\bmu$.
 Therefore, the Taylor expansion suggests the approximation
 \begin{dmath*}
 n[g(\bX_n)-g(\bmu)]
 =
 n[\nabla g(\bmu)]' (\bX_n-\bmu)
 +\frac{1}{2}[\surd{n}(\bX_n-\bmu)]'
 \BH_k(\bmu)[\surd{n}(\bX_n-\bmu)]
 +o_p(1)
 \end{dmath*}
 and, by  \eqref{sing.var}, the rhs of the
 above equals to
 \[
 \frac{1}{2}[\surd{n}(\bX_n-\bmu)]'
 \BH_k(\bmu)[\surd{n}(\bX_n-\bmu)]
 +o_p(1).
 %\distr \frac{1}{2} \bW'\BH_k(\bmu)\bW,
 \]
 By Slutsky's Theorem and in view of \eqref{lim},
 we conclude that the above quantity
 tends in distribution to $\frac{1}{2}
 \bW'\BH_k(\bmu)\bW$.
 \hfill\qed
 \end{proof}

 \Cref{lem.delta2} immediately applies
 to $M_{k,n}$ whenever the random sample arizes
 from a singular distribution.
 This result is stated in the following theorem;
 for the proof see \Cref{appendix}.

 \begin{theorem}
 \label{theo.singular.limit}
 If $M_{k,n}$ is the sample central moment of a singular
 distribution of order $k\geq 2$, then
 \begin{equation}
 \label{lim.sing}
 n(M_{k,n}-\mu_k)\distr \frac{1}{2}k(k-1)\mu_{k-2} W_1^2 - k
 W_1 W_{k-1},
 \end{equation}
 where
 \begin{equation}
 \label{sing.normal}
 \left({W_1\atop W_{k-1}}\right)
 \sim
 N_2
 \left(\left({0\atop0}\right),
 \left({\sigma^2\atop\mu_k}~~{\mu_k\atop\mu_{2k-2}-\mu_{k-1}^2}\right)
 \right).
 \end{equation}
 \end{theorem}

 The limiting distribution in \eqref{lim.sing} can be
 expressed in terms of two
 independent and identically distributed standard normal random
 variables, $Z_1$, $Z_2$. Indeed,
 observing that $\sigma^2(\mu_{2k-2}-\mu_{k-1}^2)-\mu_k^2
 =\Var[\sigma (X-\mu)^{k-1}-\mu_k {(X-\mu)}/{\sigma}]\geq
 0$,
 it is easily seen that
 %\begin{equation}
 %\label{a_k}
 \[
 \left({W_1 \atop W_{k-1}}\right)
 \law
 \left({\sigma Z_1 \atop \frac{\mu_k}{\sigma}Z_{1}+\frac{\gamma_k}{\sigma}  Z_2}\right),
 \quad \textrm{where}\quad
 \gamma_k\doteq\sqrt{\sigma^2(\mu_{2k-2}-\mu_{k-1}^2)-\mu_k^2}.
 \]
 %\end{equation}
 %observe that $\sigma^2(\mu_{2k-2}-\mu_{k-1}^2)-\mu_k^2
 %=\Var[\sigma (X-\mu)^{k-1}-\mu_k \frac{X-\mu}{\sigma}]\geq 0$
 Therefore, \eqref{lim.sing} can be rewritten as
 \begin{equation}
 \label{lim2.sing}
 n(M_{k,n}-\mu_k)\distr
 \left(\frac{1}{2}k(k-1)\sigma^2\mu_{k-2}-k\mu_k\right)
 Z_1^2 - k\gamma_k Z_1 Z_2.
 \end{equation}
 In order to obtain a further simplification, we shall make
 use of the following proposition.
 \begin{proposition}
 \label{prop.normal} If $Z_1$, $Z_2$ are
 independent and identically distributed standard
 normal random variables, then, for arbitrary constants
 $\alpha,\beta\in\R$,
 \begin{equation}
 \label{norm.simplified}
 \alpha Z_1^2+\beta Z_1Z_2 \law
 \frac{1}{2}\left(\sqrt{\alpha^2+\beta^2}+\alpha\right)Z_1^2
 -
 \frac{1}{2}\left(\sqrt{\alpha^2+\beta^2}-\alpha\right)Z_2^2.
 \end{equation}
 \end{proposition}
 \begin{proof}
 The assertion is obvious if $\beta=0$.
 Assume that $\beta\neq 0$ and set $\rho=\sqrt{\alpha^2+\beta^2}>0$.
 It is easily seen that the moment
 generating function of the rhs of
 \eqref{norm.simplified} is given by
 \[
 M_2(t)=\frac{1}{\sqrt{1-2\alpha t-\beta^2t^2}}
 \]
 and it is finite in the interval
 $\{t\in\R\colon 1-2\alpha t-\beta^2 t^2>0\}
 =(-(\rho-\alpha)^{-1},(\rho+\alpha)^{-1})$,
 which contains the origin because $\rho+\alpha>0$ and
 $\rho-\alpha>0$. Also, the moment generating function of
 the lhs of \eqref{norm.simplified} is
 \[
 M_1(t)=\E\left[\exp\left(\alpha t Z_1^2 +\beta t Z_1 Z_2\right)\right]=\frac{1}{2\pi}
 \iint_{\R^2}
 e^{-\frac{1}{2}\gamma(x,y)} \ud{y}\ud{x},
 \]
 where
 \[
 \gamma(x,y)=x^2+y^2-2\alpha t x^2 -2\beta t xy=\left(1-2\alpha t-\beta^2
 t^2\right)x^2+(y-\beta t x)^2.
 \]
 Therefore, for $t\in(-(\rho-\alpha)^{-1},(\rho+\alpha)^{-1})$,
\begin{dmath*}
 M_1(t)
 =
 \frac{1}{\sqrt{2\pi}}\int_{-\infty}^{\infty}
 e^{-\frac{x^2}{2}\left(1-2\alpha t-\beta^2 t^2\right)}
 \left(\frac{1}{\sqrt{2\pi}}\int_{-\infty}^{\infty}
 e^{-\frac{1}{2}(y-\beta t x)^2}
 \ud{y}\right)\ud{x}
 =
 {\frac{1}{\sqrt{2\pi}}\int_{-\infty}^{\infty}
 e^{-\frac{x^2}{2}\left(1-2\alpha t-\beta^2 t^2\right)} \ud{x}=M_2(t),}
 \end{dmath*}
 and the proof is complete.
 \hfill\qed
 \end{proof}

 \begin{corollary}
 \label{cor.simple}
 If $(X_1,X_2)'$ follows a bivariate normal distribution
 with $\E(X_1)=\mu_1$, $\E(X_2)=\mu_2$, $\Var(X_1)=\sigma_1^2$, $\Var(X_2)=\sigma_2^2$
 and $\Cov(X_1,X_2)=\rho\sigma_1\sigma_2$, where
 $\mu_1,\mu_2\in\R$ and
 $\sigma_1\geq 0$, $\sigma_2\geq 0$ and $-1\leq\rho\leq 1$
 are arbitrary constants, then
 \[
  (X_1-\mu_1)(X_2-\mu_2) \law
 \sigma_1\sigma_2 \left[
 \frac{1}{2}(1+\rho)Z_1^2-\frac{1}{2}(1-\rho)Z_2^2\right].
 \]
 \end{corollary}
 \begin{proof}
 Since
 $(X_1-\mu_1,X_2-\mu_2)'\law(\sigma_1 Z_1,
 \sigma_2(\rho Z_{1}+ \sqrt{1-\rho^2} Z_2))'$,
 %\[
 %\left({\ds X_1-\mu_1 \atop \ds X_2-\mu_2}\right)
 %\law
 %\left(\begin{array}{c} \sigma_1 Z_1 \\
 %\sigma_2\left(\rho Z_{1}+ \sqrt{1-\rho^2} Z_2\right)
 %\end{array} \right),
 %\]
 we have that
 $(X_1-\mu_1)(X_2-\mu_2)\law\sigma_1\sigma_2 (\rho Z_1^2+ \sqrt{1-\rho^2}
 Z_1Z_2)$,
 and the assertion follows from \eqref{norm.simplified}
 with $\alpha=\rho$ and $\beta=\sqrt{1-\rho^2}$.
 \hfill\qed
 \end{proof}

 The main result is contained in the following theorem; its
 proof, being an immediate consequence of
 \eqref{lim2.sing} and \Cref{prop.normal}, is omitted.
 \begin{theorem}
 \label{theo.singular.limit2}
 If $M_{k,n}$ is the sample central moment of a singular
 distribution of order $k\geq 2$, then
 %\begin{equation}
 %\label{lim.sing2}
 \[
 n(\mu_k-M_{k,n})\distr \frac{k}{2}\left(\sigma
 \surd{\theta_k}+\alpha_k\right)
 Z_1^2
 - \frac{k}{2}\left(\sigma\surd{\theta_k}-\alpha_k\right)
 Z_2^2,
 \]
 %\end{equation}
 where $Z_1$, $Z_2$ are independent and identically distributed standard normal and
 \begin{equation}
 \label{a_k.b_k}
 \begin{split}
 \alpha_k&=\mu_k-\frac{1}{2}(k-1)\sigma^2 \mu_{k-2}, \\
 \theta_k&=\mu_{2k-2}-\mu_{k-1}^2-(k-1)\mu_{k-2}
 \left[\mu_k-\frac{1}{4}(k-1)\sigma^2\mu_{k-2}\right].
 \end{split}
 \end{equation}
 \end{theorem}

 \begin{corollary}
 \label{cor.singular.limit}
 If $M_{k,n}$ is the sample central moment of a singular
 distribution of order $k\geq 2$, then
 there exists a constant $\lambda_k\in\R$
 such that
 \[
 n(\mu_k-M_{k,n})\distr \lambda_k \chi_1^2
 \]
 if and only if
 %\begin{itemize}
 %\item[(i)]
 %$n(\mu_k-M_{k,n})\distr C_k \chi_1^2$
 %, as $n\to\infty$,
 \begin{equation}
 \label{sing.0}
 \mu_k^2=\sigma^2\left(\mu_{2k-2}-\mu_{k-1}^2\right).
 \end{equation}
 If \eqref{sing.0} holds,
 $\lambda_k=k(\mu_k-({1}/{2})(k-1)\sigma^2\mu_{k-2})$
 and, thus,
 \begin{equation}
 \label{sing.1}
 n(\mu_k-M_{k,n})\distr k\left[\mu_k-\frac{1}{2}(k-1)\sigma^2\mu_{k-2}\right]
 \chi_1^2.
 \end{equation}
 %\end{itemize}
 If \eqref{sing.0} does not hold,
 \begin{equation}
 \label{sing.2}
 n(\mu_k-M_{k,n})\distr \lambda_k \chi_1^2 - \tilde{\lambda}_k\tilde{\chi}_1^2
 \end{equation}
 with $\lambda_k=({k}/{2})(\sigma
 \surd{\theta_k}+\alpha_k)>0$,
 $\tilde{\lambda}_k=({k}/{2})(\sigma
 \surd{\theta_k}-\alpha_k)>0$,
 $\alpha_k$ and $\theta_k$ as in \eqref{a_k.b_k},
 and where $\chi_1^2$ and $\tilde{\chi}_1^2$ are
 independent and identically distributed
 random variables from the chi-square distribution
 with one
 degree of freedom.
 \end{corollary}

 After some algebra it follows that \eqref{sing.0} is
 satisfied by all two-valued distributions of
 \Cref{cor.even,cor.odd}. In particular,
 from \eqref{sing.1} we can show that
 \[
 n(\mu_k-M_{k,n})\distr \frac{k(k-1)}{2} p_k^{k-1}
 \frac{(k+1)p_k^2-(k+1)p_k+1}{(k+1)p_k-1}\chi_1^2, \quad
 k=3,4,\ldots~.
 \]
 For example, the two-valued standardized distribution
 of \Cref{cor.even}
 with $p_6={1}/{2}+\sqrt{15(4\surd{10}-5)}/30$
 has sixth central moment equal to
 $\mu_6=(50-13\surd{10})/45$ and
 \[
 n\left( \frac{4\surd{10}-5}{135}-M_{6,n}\right)\distr
 \frac{50-13\surd{10}}{45} \chi_1^2.
 \]
 Finally, for the symmetric distributions of
 \Cref{lem.sym} one finds that \eqref{sing.0} is
 not satisfied and that
 $\lambda_k=\tilde{\lambda}_k=({\sqrt{k-1}}/{2})k^{{(k-1)}/{2}}$.
 Hence, since $\mu_k=0$, we conclude from \eqref{sing.2}
 the limit
 \[
 n M_{k,n}\distr \frac{\sqrt{k-1}}{2}k^{{(k-1)}/{2}}
 \left(\chi_1^2-\tilde{\chi}_1^2\right), \quad
 k=3,5,7,\ldots~.
 \]

 %\section{Discussion}
 %\label{sec.discussion}

 \begin{appendix}
 {
 \normalsize
 \section{Proofs}
 \label{appendix}

 We shall make use of the following Lemmas.
 For the proof of \Cref{lem.ui} see, e.g., \citet[p.~18]{Gut1988};
 for more general results, see \cite{AM2016}.

 \begin{lemma}
 \label{lem.ui}
 If $X,X_1,\ldots,X_n$ are independent and identically distributed
 with $\E(X)=\mu$, $\Var(X)=\sigma^2$
 and $\E|X|^\delta<\infty$ for some $\delta\geq 2$, then, for any
 $\alpha\in(0,\delta]$,
 %\begin{equation}
 %\label{ui}
 \[
 \E\left|\surd{n}\left(\bar{X}_n-\mu\right)\right|^\alpha \to \sigma^\alpha \E|Z|^\alpha,
 \]
 %\end{equation}
 where $Z\sim N(0,1)$ and
 $\bar{X}_n=(X_1+\cdots+X_n)/n$.
 \end{lemma}
 %\begin{proof}
 %See, e.g., \citet[p.~18]{Gut1988}.
 %\hfill\qed
 %\end{proof}

 \begin{lemma}
 \label{lem.holder}
 If $X,X_1,\ldots,X_n$ are independent and identically distributed
 with $\E(X)=\mu$ and $\E|X|^\nu<\infty$ for some $\nu\in\{2,3,\ldots\}$,
 then, for any $j\in\{2,\ldots,\nu\}$,
 \begin{equation}
 \label{holder}
 \E|m_{j,n}|^{{\nu}/{j}} \leq \E|X-\mu|^{\nu},
 \end{equation}
 where
 $m_{j,n}=n^{-1}\sum_{i=1}^n(X_i-\mu)^j$.
 \end{lemma}
 \begin{proof}
 If $j=\nu$, then \eqref{holder} follows by taking expectations to the obvious
 inequality $|m_{j,n}|\leq\frac{1}{n}\sum_{i=1}^n|X_i-\mu|^j=\frac{1}{n}
 \sum_{i=1}^n|X_i-\mu|^{\nu}$.
 If $j<\nu$ (and thus, $\nu\geq 3$), we apply
 the inequality
 \[
 \left|\sum_{i=1}^{n} x_i\right|^p
 \leq \left(\sum_{i=1}^{n} |x_i|\right)^p
  \leq n^{p-1} \sum_{i=1}^n
 |x_i|^p,
 \quad p>1,
 \]
 (the last inequality is a by-product of H\"{o}lder's inequality)
 for $p=\nu/j$ and $x_i=(X_i-\mu)^j$. Then, we have
 \begin{align}
  \E|m_{j,n}|&^{{\nu}/{j}}
  = \frac{1}{n^{{\nu}/{j}}}
 \E\left|\sum_{i=1}^n (X_i-\mu)^j\right|^{{\nu}/{j}}
 \leq
 \frac{1}{n^{{\nu}/{j}}}
 \E\left(\sum_{i=1}^n |X_i-\mu|^j\right)^{{\nu}/{j}}
 \nonumber\\
  &\leq
  \frac{1}{n^{{\nu}/{j}}}
 \E\left( n^{{\nu}/{j}-1}
 \sum_{i=1}^n |X_i-\mu|^\nu
 \right)
 =\frac{1}{n}
 \E\left(\sum_{i=1}^n |X_i-\mu|^\nu
 \right)=\E|X-\mu|^{\nu}.
  \tag*{\qed}
 \end{align}
 \end{proof}

 \begin{pr}{Proof of \Cref{prop.exp.cov.conv}}
 \eqref{prop.exp.cov.conv(a)}
 Observe that the statement in \Cref{prop.exp.cov.conv}\eqref{prop.exp.cov.conv(a)} is equivalent to
 \begin{equation}
 \label{exp.conv2}
 \E[\surd{n}(M_{k,n}-\mu_k)]\to 0.
 \end{equation}
 Writing
 \begin{equation}
 \label{central}
 M_{k,n}-\mu_k=
 (m_{k,n}-\mu_k)+(-1)^{k-1}(k-1) m_{1,n}^k
 +\sum_{j=2}^{k-1} (-1)^{k-j} {k\choose j}
  m_{1,n}^{k-j} m_{j,n},
 \end{equation}
 it suffices to verify that
 \begin{enumerate}[leftmargin=17pt, itemsep=.5ex, labelindent=0pt, label=\rm(\roman*), ref=\textcolor{black}{\roman*}]
 \item
 $\surd{n}\E(m_{k,n}-\mu_k)=0$,
 \item
 $\surd{n}\E(m_{1,n}^k)\to 0$,
 \item
 $\surd{n}\E (m_{1,n}^{k-j} m_{j,n}) \to 0$, $j=2,\ldots,k-2$
 (provided $k\geq 4$), and
 \item
 $\surd{n}\E (m_{1,n} m_{k-1,n}) \to 0$ (provided $k\geq 3$).
 \end{enumerate}
 Now, (i) is obvious (since $\E(m_{k,n})=\mu_k$), (iv)
 follows  from $\E(m_{1,n} m_{k-1,n})=\mu_k/n$ and
 (ii) can be seen by using \Cref{lem.ui} with
 $\alpha=\delta=k$, which shows that
 \[
 \left|n^{k/2}\E\left(m_{1,n}^k\right)\right|\leq n^{k/2}
 \E |m_{1,n}|^k=\E\left|\surd{n}\left(\bar{X}_n-\mu\right)\right|^k\to
 \sigma^{k}\E|Z|^k<\infty,
 \]
 and thus,
 $|\surd{n}\E(m_{1,n}^k)|\leq n^{-(k-1)/2}\E|\surd{n}(\bar{X}_n-\mu)|^k\to
 0$. To show (iii), we assume that $k\geq 4$ and $2\leq j\leq k-2$, and
 we use H\"{o}lder's inequality with $p=k/(k-j)>1$, \Cref{lem.holder}
 with $\nu=k$ and \Cref{lem.ui} with $\alpha=\delta=k$
 to obtain
 \begin{dmath*}
 \left|\surd{n}\E\left(m_{1,n}^{k-j}m_{j,n}\right)\right|
 \leq
 {\surd{n}\E\left(|m_{1,n}|^{k-j} |m_{j,n}|\right)\leq
 \surd{n}\left(\E|m_{1,n}|^k\right)^{{(k-j)}/{k}}
 \left(\E|m_{j,n}|^{{k}/{j}}\right)^{{j}/{k}}
 }
 \leq
 \surd{n}
 \left[n^{-k/2}\E\left|\surd{n}\left(\bar{X}_n-\mu\right)\right|^k\right]^{{(k-j)}/{k}}
 \left(\E|X-\mu|^k\right)^{{j}/{k}}
 ={n^{-(k-1-j)/2}
 \left[\E\left|\surd{n}\left(\bar{X}_n-\mu\right)\right|^k\right]^{{(k-j)}/{k}}
 \left(\E|X-\mu|^k\right)^{{j}/{k}}
  }
  =
  {n^{-(k-1-j)/2} O(1)\to 0,}
 \end{dmath*}
 because $\E|\surd{n}(\bar{X}_n-\mu)|^k\to
 \sigma^{k}\E|Z|^k<\infty$.
 \medskip

 \noindent
 \eqref{prop.exp.cov.conv(b)}
 Observe that the statement in \Cref{prop.exp.cov.conv}\eqref{prop.exp.cov.conv(b)} is equivalent to
 \begin{equation}
 \label{cov.conv2}
 \Cov\left[\surd{n}\left(\bar{X}_n-\mu\right),\surd{n}(M_{k,n}-\mu_k)\right]\to
 \mu_{k+1}-k\sigma^2\mu_{k-1},
 \end{equation}
 and since $\E(\bar{X}_n-\mu)=0$, it suffices to verify that
 \begin{equation}
 \label{cov.conv3}
 n\E\left[\left(\bar{X}_n-\mu\right)(M_{k,n}-\mu_k)\right]
 =n\E[m_{1,n}(M_{k,n}-\mu_k)]\to
 \mu_{k+1}-k\sigma^2\mu_{k-1}.
 \end{equation}

 If $k=2$, then
 $n\E[m_{1,n}(M_{2,n}-\mu_2)]=n\E[(\bar{X}_n-\mu)(m_{2,n}-\mu_2)]
 -n\E(\bar{X}_n-\mu)^3=\mu_3-n\E(\bar{X}_n-\mu)^3$,
 and it easily seen, by \Cref{lem.ui} with
 $\alpha=\delta=3$, that $|n\E(\bar{X}_n-\mu)^3|\linebreak\leq n^{-1/2}
 \E|\surd{n}(\bar{X}_n-\mu)|^3\to 0$; thus, $n\E[m_{1,n}(M_{2,n}-\mu_2)]\to
 \mu_3$. Since $\mu_1=0$, \eqref{cov.conv3} is satisfied for $k=2$.

 If $k=3$,
 $n\E[m_{1,n}(M_{3,n}-\mu_3)]=n\E[(\bar{X}_n-\mu)(m_{3,n}-\mu_3)]
 +2n\E(\bar{X}_n-\mu)^4-3n\linebreak\E[m_{2,n}(\bar{X}_n-\mu)^2]$,
 and it is easy to see that
 $n\E[(\bar{X}_n-\mu)(m_{3,n}-\mu_3)]=\mu_4$. Also,
 by \Cref{lem.ui} with $\alpha=\delta=4$, $2n\E(\bar{X}_n-\mu)^4\to
 0$. Finally, $-3n\E[m_{2,n}(\bar{X}_n-\mu)^2]=-3[\mu_4+(n-1)\mu_2^2]/n\to -3
 \mu_2^2=-3\sigma^4$, which verifies \eqref{cov.conv3} for
 $k=3$.

 In the general case when $k\geq 4$, we write $M_{k,n}-\mu_k$ as in
 \eqref{central} and we observe that
 for \eqref{cov.conv3} to hold
 it suffices to verify that
 \begin{enumerate}[leftmargin=17pt, itemsep=.5ex, labelindent=0pt, label=\rm(\roman*), ref=\textcolor{black}{\roman*}]
 \item
 $n \E[m_{1,n}(m_{k,n}-\mu_k)]=\mu_{k+1}$,
 \item
 $n\E(m_{1,n}^{k+1}) \to 0$,
 \item
 $n \E(m_{1,n}^{k+1-j} m_{j,n}) \to 0$, $j=2,\ldots,k-2$, and
 \item
 $n\E(m_{1,n}^2 m_{k-1,n}) \to \sigma^2 \mu_{k-1}$.
 \end{enumerate}
 Calculating $\E[m_{1,n}(m_{k,n}-\mu_k)]=\E[(\bar{X}_n-\mu)(m_{k,n}-\mu_k)]
 =\E[(\bar{X}_n-\mu)m_{k,n}]
 =\linebreak n^{-2}\sum_{i_1=1}^n\sum_{i_2=1}^n
  \E[(X_{i_1}-\mu)(X_{i_2}-\mu)^k]
 ={\mu_{k+1}}/{n}$, we conclude (i),
 while (ii) follows by using \Cref{lem.ui} with
 $\alpha=\delta=k+1$. Also,
 \begin{dmath*}
 n\E\left(m_{1,n}^2 m_{k-1,n}\right)
 =
 \frac{1}{n^2}\sum_{i_1=1}^n\sum_{i_2=1}^n\sum_{i_3=1}^n
 \E\left[\left(X_{i_1}-\mu\right)\left(X_{i_2}-\mu\right)\left(X_{i_3}-\mu\right)^{k-1}\right]
 =
 {
 \frac{1}{n^2}\left[n \mu_{k+1}+n(n-1)\sigma^2\mu_{k-1}\right]
 \to \sigma^2\mu_{k-1},
 }
 \end{dmath*}
 which shows that (iv) is satisfied, and it remains to verify (iii).
 To this end, we use H\"{o}lder's
 inequality with $p=(k+1)/(k+1-j)>1$
 and \Cref{lem.holder} with $\nu=k+1$
 to obtain
 \begin{dmath*}
 \left|n \E\left(m_{1,n}^{k+1-j}m_{j,n}\right)\right|
 \leq
 {
 n
 \E|m_{1,n}|^{k+1-j} |m_{j,n}|
 \leq
 n \left(\E|m_{1,n}|^{k+1}\right)^{{(k+1-j)}/{(k+1)}}
 \left(\E|m_{j,n}|^{(k+1)/{j}}\right)^{{j}/{(k+1)}}
 }
 \leq
 {
 n
 \left[n^{-(k+1)/2}\E\left|\surd{n}\left(\bar{X}_n-\mu\right)\right|^{k+1}\right]^{{(k+1-j)}/{(k+1)}}
 \left(\E|X-\mu|^{k+1}\right)^{{j}/{(k+1)}}
 }
 =
 {
 n^{-(k-1-j)/2}
 \left[\E\left|\surd{n}\left(\bar{X}_n-\mu\right)\right|^{k+1}\right]^{{(k+1-j)}/{(k+1)}}
 \left(\E|X-\mu|^{k+1}\right)^{{j}/{(k+1)}}
  %=n^{-(k-1-j)/2} \times O(1)
 \to 0,
 }
 \end{dmath*}
 because $n^{-(k-1-j)/2}\to 0$; and, by
 \Cref{lem.ui} with $\alpha=\delta=k+1$,
 $\E|\surd{n}(\bar{X}_n-\mu)|^{k+1}\to \sigma^{k+1}\E|Z|^{k+1}
 <\infty$.
 \medskip

 \noindent
 \eqref{prop.exp.cov.conv(c)}
 Without loss of generality assume that $2\leq r\leq k$
 and observe that the first statement of
 \Cref{prop.exp.cov.conv}\eqref{prop.exp.cov.conv(c)}
 is equivalent to
 \begin{equation}
 \label{covar.conv2}
 \Cov[\surd{n}(M_{r,n}-\mu_r),\surd{n}(M_{k,n}-\mu_k)]\to
 v_{rk}.
 \end{equation}
 Since $\E|X|^{r+k}<\infty$, \eqref{exp.conv2} shows that
 $\E[\surd{n}(M_{k,n}-\mu_k)]\to 0$ and $\E[\surd{n}(\bar{X}_n-\mu)]\to 0$,
 and it suffices to verify that
 \begin{dmath}
 \label{covar.conv3}
 n\E[(M_{r,n}-\mu_r)(M_{k,n}-\mu_k)]
 \hiderel\to v_{rk}=\mu_{r+k}-\mu_r\mu_k-r\mu_{r-1}\mu_{k+1}
 -k\mu_{r+1}\mu_{k-1}+rk\sigma^2\mu_{r-1}\mu_{k-1}.
 \end{dmath}
 The proof can be deduced by showing that \eqref{covar.conv3}
 holds for each one of the cases $r=k=2$; $r=2$, $k=3$; $r=k=3$; $r=2$,
 $k\geq4$; $r=3$, $k\geq 4$; $4\leq r\leq k$.
 In the following we shall present the details only for the case
 where $4\leq r\leq k$; the other cases can be treated
 using similar (and simpler) arguments.

 Assume now that $4\leq r\leq k$.
 From \eqref{central}, we have
 \begin{dmath}
 \label{central_r}
 M_{r,n}-\mu_r
 =
 (m_{r,n}-\mu_r)-rm_{1,n}m_{r-1,n}
 +\sum_{j_1=2}^{r-2} (-1)^{r-j_1} {r\choose j_1} m_{1,n}^{r-j_1} m_{j_1,n}
 +(-1)^{r-1}(r-1) m_{1,n}^r,
 \end{dmath}
 \begin{dmath}
 \label{central_k}
 M_{k,n}-\mu_k
 =
 (m_{k,n}-\mu_k)-km_{1,n}m_{k-1,n}
 +\sum_{j_2=2}^{k-2} (-1)^{k-j_2} {k\choose j_2} m_{1,n}^{k-j_2} m_{j_2,n}
 +(-1)^{k-1}(k-1) m_{1,n}^k.
 \end{dmath}
 We shall show that the asymptotic covariance in \eqref{covar.conv2}
 can be determined by using only the first two terms in
 \eqref{central_r} and \eqref{central_k}. Indeed, it is
 easily seen that
 \eqref{covar.conv3} holds true if it can be shown that
 \begin{enumerate}[leftmargin=20pt, itemsep=.5ex, labelindent=0pt, label=\rm(\roman*), ref=\textcolor{black}{\roman*}]
  \item
 $n \E[(m_{r,n}-\mu_r)(m_{k,n}-\mu_k)]=\mu_{r+k}-\mu_r\mu_k$,
 \item
 $n\E[m_{1,n}m_{k-1,n}(m_{r,n}-\mu_r)] \to
 \mu_{r+1}\mu_{k-1}$,
 \item
 $n\E[m_{1,n}m_{r-1,n}(m_{k,n}-\mu_k)] \to
 \mu_{r-1}\mu_{k+1}$,
 \item
 $n\E(m_{1,n}^2m_{r-1,n}m_{k-1,n}) \to
 \sigma^2 \mu_{r-1}\mu_{k-1}$,
 \item
 $n \E[m_{1,n}^{k-j_2} m_{j_2,n}(m_{r,n}-\mu_r)] \to 0$, $j_2=2,\ldots,k-2$,
 \item
 $n\E [m_{1,n}^k (m_{r,n}-\mu_r)] \to 0$,
 \item
 $n \E (m_{1,n}^{k+1-j_2} m_{j_2,n}m_{r-1,n}) \to 0$, $j_2=2,\ldots,k-2$,
 \item
 $n\E (m_{1,n}^{k+1} m_{r-1,n}) \to 0$,
 \item
 $n \E [m_{1,n}^{r-j_1} m_{j_1,n}(m_{k,n}-\mu_k)] \to 0$, $j_1=2,\ldots,r-2$,
 \item
 $n\E (m_{1,n}^{r+1-j_1} m_{j_1,n} m_{k-1,n}) \to 0$, $j_1=2,\ldots,r-2$,
 \item
 $n \E (m_{1,n}^{r+k-j_1-j_2} m_{j_1,n}m_{j_2,n}) \to 0$,
    $j_1=2,\ldots,r-2$, $j_2=2,\ldots,k-2$,
 \item
 $n\E (m_{1,n}^{r+k-j_1} m_{j_1,n}) \to 0$, $j_1=2,\ldots,r-2$,
 \item
 $n \E [m_{1,n}^{r} (m_{k,n}-\mu_k)] \to 0$,
 \item
 $n\E (m_{1,n}^{r+1} m_{k-1,n}) \to 0$,
 \item
 $n \E (m_{1,n}^{r+k-j_2} m_{j_2,n}) \to 0$,
 $j_2=2,\ldots,k-2$, and
 \item
 $n\E (m_{1,n}^{r+k}) \to 0$.
 \end{enumerate}
 We now proceed to verify (i)--(xvi).
 Since $\E(m_{r,n})=\mu_r$ and $\E(m_{k,n})=\mu_k$,
 we have
 \begin{dmath*}
 n\E[(m_{r,n}-\mu_r)(m_{k,n}-\mu_k)]
 =
 {
 n[\E(m_{r,n}m_{k,n})-\mu_r\mu_k]
 =
 n\left\{
 \frac{1}{n^2}\sum_{i_1=1}^n\sum_{i_2=1}^n
 \E\left[\left(X_{i_1}-\mu\right)^r\left(X_{i_2}-\mu\right)^k\right]-\mu_r\mu_k\right\}
 }
 =
 {
 n\left\{\frac{1}{n^2}[n\mu_{r+k}+n(n-1)\mu_r\mu_k]-\mu_r\mu_k\right\}
 =\mu_{r+k}-\mu_r\mu_k,
 }
 \end{dmath*}
 which shows (i). Also, (ii), (iii) and (iv)
 follow by straightforward computations; e.g., for (ii) we have
 \begin{dmath*}
 n\E[m_{1,n}m_{k-1,n}(m_{r,n}-\mu_r)]
 =-\mu_r\mu_k
 +\frac{\mu_{r+k}+(n-1)(\mu_{r+1}\mu_{k-1}+
 \mu_{r}\mu_{k})}{n}\to\mu_{r+1}\mu_{k-1},
 \end{dmath*}
 while (iii) is similar to (ii), and (iv) can be deduced
 from
 \[
 n\E\left(m_{1,n}^2m_{r-1,n}m_{k-1,n}\right)
 =\frac{1}{n^3}\left[n(n-1)(n-2)\sigma^2\mu_{r-1}\mu_{k-1} +o\left(n^3\right)\right]
 \to\sigma^2\mu_{r-1}\mu_{k-1}.
 \]

 The vanishing limits (vi)--(viii) and (x)--(xvi)
 are by-products of \Cref{lem.ui,lem.holder}
 with $\alpha=\delta=\nu=r+k$, since
 $\E|X|^{r+k}<\infty$. Indeed, we have
 $|n\E(m_{1,n}^{r+k})|\leq n\E|m_{1,n}|^{r+k}
 =n^{-(r+k-2)/2}\E|\surd{n}(\bar{X}_n-\mu)|^{r+k}\to
 0$, which verifies (xvi). Also, using H\"{o}lder's
 inequality with $p=(r+k)/(r+k-j_2)>1$,
 we obtain (xv) as follows:
 \begin{dmath*}
 \left|n \E\left(m_{1,n}^{r+k-j_2}m_{j_2,n}\right)\right|
 \leq
 {n \E\left(|m_{1,n}|^{r+k-j_2} |m_{j_2,n}|\right)
 \leq
 n \left(\E|m_{1,n}|^{r+k}\right)^{\frac{r+k-j_2}{r+k}}
 \left(\E|m_{j_2,n}|^{\frac{r+k}{j_2}}\right)^{\frac{j_2}{r+k}}}
 \leq
 n\left[n^{-(r+k)/2}\E\left|\surd{n}\left(\bar{X}_n-\mu\right)\right|^{r+k}\right]^{\frac{r+k-j_2}{r+k}}
 \left(\E|X-\mu|^{r+k}\right)^{\frac{j_2}{r+k}}
 ={n^{-(r+k-j_2-2)/2}
 \left[\E\left|\surd{n}\left(\bar{X}_n-\mu\right)\right|^{r+k}\right]^{\frac{r+k-j_2}{r+k}}
 \left(\E|X-\mu|^{r+k}\right)^{\frac{j_2}{r+k}}
  %=n^{-(k-1-j)/2} \times O(1)
 \to 0,}
 \end{dmath*}
 because $n^{-(r+k-j_2-2)/2}\to 0$ and
 $\E|\surd{n}(\bar{X}_n-\mu)|^{r+k}\to \sigma^{r+k}\E|Z|^{r+k}
 <\infty$; (xii) is similar to (xv). For the limit (xiv)
 we have
 \begin{dmath*}
 \left|n\E \left[m_{1,n}^{r+1}m_{k-1,n}\right]\right|
 \leq
 {
 n\E\left(|m_{1,n}|^{r+1} |m_{k-1,n}|\right)
 \leq
 n \left(\E|m_{1,n}|^{r+k}\right)^{\frac{r+1}{r+k}}
 \left(\E|m_{k-1,n}|^{\frac{r+k}{k-1}}\right)^{\frac{k-1}{r+k}}
 }
 \leq
 {
 n^{-(r-1)/2}\left[\E\left|\surd{n}\left(\bar{X}_n-\mu\right)\right|^{r+k}\right]^{\frac{r+1}{r+k}}
 \left(\E|X-\mu|^{r+k}\right)^{\frac{k-1}{r+k}}
 \to 0,
 }
 \end{dmath*}
 and similarly for (viii). In order to prove (xiii),
 it is sufficient to show that
 $n\E[m_{1,n}^r m_{k,n}]\to 0$ and $n\E(m_{1,n}^r)\to 0$.
 The second limit is obvious since, as for (xvi), one can
 easily verify that
 $|n\E(m_{1,n}^r)|\leq n^{-(r-2)/2}\E|\surd{n}(\bar{X}_n-\mu)|^r
 =n^{-(r-2)/2}O(1)\to 0$. For the first limit, we have
 \begin{dmath*}
 \left|n\E \left(m_{1,n}^{r}m_{k,n}\right)\right|
 \leq
 {
 n\E\left(|m_{1,n}|^{r} |m_{k,n}|\right)
 \leq
 n \left(\E|m_{1,n}|^{r+k}\right)^{\frac{r}{r+k}}
 \left(\E|m_{k,n}|^{\frac{r+k}{k}}\right)^{\frac{k}{r+k}}
 }
 \leq
 {
 n^{-(r-2)/2}\left[\E\left|\surd{n}\left(\bar{X}_n-\mu\right)\right|^{r+k}\right]^{\frac{r}{r+k}}
 \left(\E|X-\mu|^{r+k}\right)^{\frac{k}{r+k}}
 \to 0.
 }
 \end{dmath*}
 Limit (vi) is similar to (xiii) and its proof is omitted.
 Regarding (xi), we have
 \begin{dmath*}
 \left|n\E \left(m_{1,n}^{r+k-j_1-j_2}m_{j_1,n}m_{j_2,n}\right)\right|
 \leq
 n\E\left(|m_{1,n}|^{r+k-j_1-j_2} |m_{j_1,n}m_{j_2,n}|\right)
 \leq
 n \left(\E|m_{1,n}|^{r+k}\right)^{\frac{r+k-j_1-j_2}{r+k}}
 \left(\E|m_{j_1,n}m_{j_2,n}|^{\frac{r+k}{j_1+j_2}}\right)^{\frac{j_1+j_2}{r+k}}
 \leq
 n \left(\E|m_{1,n}|^{r+k}\right)^{\frac{r+k-j_1-j_2}{r+k}}
 \left[\left(\E|m_{j_1,n}|^{\frac{r+k}{j_1}}\right)^{\frac{j_1}{j_1+j_2}}
 \left(\E|m_{j_2,n}|^{\frac{r+k}{j_2}}\right)^{\frac{j_2}{j_1+j_2}}\right]^{\frac{j_1+j_2}{r+k}}
 \leq
 n \left(\E|m_{1,n}|^{r+k}\right)^{\frac{r+k-j_1-j_2}{r+k}}
 \left(\E|X-\mu|^{r+k}\right)^{\frac{j_1+j_2}{r+k}}
 =
 n^{-(r+k-j_1-j_2-2)/2} \left[\E\left|\surd{n}\left(\bar{X}_n-\mu\right)\right|^{r+k}\right]^{\frac{r+k-j_1-j_2}{r+k}}
 \left(\E|X-\mu|^{r+k}\right)^{\frac{j_1+j_2}{r+k}} \hiderel\to 0.
 \end{dmath*}
 Similarly, for (x) we have
 \begin{dmath*}
 \left|n\E \left(m_{1,n}^{r+1-j_1}m_{j_1,n}m_{k-1,n}\right)\right|
 \leq
 n\left(\E|m_{1,n}|^{r+1-j_1} |m_{j_1,n}m_{k-1,n}|\right)
 \leq
 n \left(\E|m_{1,n}|^{r+k}\right)^{\frac{r+1-j_1}{r+k}}
 \left(\E|m_{j_1,n}m_{k-1,n}|^{\frac{r+k}{j_1+k-1}}\right)^{\frac{j_1+k-1}{r+k}}
 \leq
 n \left(\E|m_{1,n}|^{r+k}\right)^{\frac{r+1-j_1}{r+k}}
 \left[\left(\E|m_{j_1,n}|^{\frac{r+k}{j_1}}\right)^{\frac{j_1}{j_1+k-1}}
 \left(\E|m_{k-1,n}|^{\frac{r+k}{k-1}}\right)^{\frac{k-1}{j_1+k-1}}\right]^{\frac{j_1+k-1}{r+k}}
 \leq
 n \left(\E|m_{1,n}|^{r+k}\right)^{\frac{r+1-j_1}{r+k}}
 \left(\E|X-\mu|^{r+k}\right)^{\frac{j_1+k-1}{r+k}}
 =
 n^{-(r-j_1-1)/2} \left[\E\left|\surd{n}\left(\bar{X}_n-\mu\right)\right|^{r+k}\right]^{\frac{r+1-j_1}{r+k}}
 \left(\E|X-\mu|^{r+k}\right)^{\frac{j_1+k-1}{r+k}} \hiderel\to 0,
 \end{dmath*}
 while (vii) is similar to (x).

 It remains to verify (v) and (ix); but, since they are similar,
 it suffices to prove (v). If $j_2\in\{2,\ldots,k-3\}$ (and
 hence, $k\geq 5$ and $j_2<k-2$), we have
 \begin{dmath*}
 \left|n\E \left[m_{1,n}^{k-j_2}m_{j_2,n}(m_{r,n}-\mu_r)\right]\right|
 \leq
 n\E\left(|m_{1,n}|^{k-j_2}|m_{j_2,n}m_{r,n}|\right)+
 n |\mu_r| \E\left(|m_{1,n}|^{k-j_2}|m_{j_2,n}|\right),
 \end{dmath*}
 and it suffices to prove that $n\E(|m_{1,n}|^{k-j_2}|m_{j_2,n}m_{r,n}|)\to 0$
 and $n \E(|m_{1,n}|^{k-j_2}|m_{j_2,n}|)\to 0$.
 For the first quantity, we have
 \begin{dmath*}
 n\E\left(\left|m_{1,n}\right|^{k-j_2}\left|m_{j_2,n}m_{r,n}\right|\right)
 \leq
 n \left(\E|m_{1,n}|^{r+k}\right)^{\frac{k-j_2}{r+k}}
 \left(\E|m_{j_2,n}m_{r,n}|^{\frac{r+k}{r+j_2}}\right)^{\frac{r+j_2}{r+k}}
 \leq
 n \left(\E|m_{1,n}|^{r+k}\right)^{\frac{k-j_2}{r+k}}
 \left[\left(\E|m_{j_2,n}|^{\frac{r+k}{j_2}}\right)^{\frac{j_2}{r+j_2}}
 \left(\E|m_{r,n}|^{\frac{r+k}{r}}\right)^{\frac{r}{r+j_2}}\right]^{\frac{r+j_2}{r+k}}
 \leq
 n^{-(k-j_2-2)/2} \left[\E\left|\surd{n}\left(\bar{X}_n-\mu\right)\right|^{r+k}\right]^{\frac{k-j_2}{r+k}}
 \left(\E|X-\mu|^{r+k}\right)^{\frac{r+j_2}{r+k}} \hiderel\to 0,
 \end{dmath*}
 because $k-j_2-2>0$. Similarly, for the second quantity we
 have
 \begin{dmath*}
 n\E\left(\left|m_{1,n}\right|^{k-j_2}\left|m_{j_2,n}\right|\right)
 \leq
 n \left(\E|m_{1,n}|^{r+k}\right)^{\frac{k-j_2}{r+k}}
 \left(\E|m_{j_2,n}|^{\frac{r+k}{r+j_2}}\right)^{\frac{r+j_2}{r+k}}
 \leq
 n \left(\E|m_{1,n}|^{r+k}\right)^{\frac{k-j_2}{r+k}}
 \left(\E|m_{j_2,n}|^{\frac{r+k}{j_2}}\right)^{\frac{j_2}{r+k}}
 \leq
 n^{-(k-j_2-2)/2} \left[\E\left|\surd{n}\left(\bar{X}_n-\mu\right)\right|^{r+k}\right]^{\frac{k-j_2}{r+k}}
 \left(\E|X-\mu|^{r+k}\right)^{\frac{j_2}{r+k}} \hiderel\to 0,
 \end{dmath*}
 because $k-j_2-2>0$.
 Finally, it remains to study the limit (v) when $j_2=k-2$;
 in this case the above limits do not necessarily vanish.
 However, since $j_2=k-2$ we have
 \[
 n\E \left[m_{1,n}^{k-j_2}m_{j_2,n}(m_{r,n}-\mu_r)\right]
 =n\E\left(m_{1,n}^2 m_{r,n} m_{k-2,n}\right)
  -n\mu_r \E\left(m_{1,n}^2m_{k-2,n}\right),
 \]
 and direct computations show that
 \[
 n\E\left(m_{1,n}^2 m_{r,n}
 m_{k-2,n}\right)=\frac{1}{n^3}\left[n(n-1)(n-2)
 \sigma^2\mu_r\mu_{k-2}+o\left(n^3\right)\right]\to
 \sigma^2\mu_r\mu_{k-2}
 \]
 and
 \[
 n\E\left(m_{1,n}^2 m_{k-2,n}\right)=\frac{1}{n^2}\left[n(n-1)
 \sigma^2\mu_{k-2}+o\left(n^2\right)\right]\to
 \sigma^2\mu_{k-2}.
 \]
 Hence, when $j_2=k-2$ we have
 \begin{dmath*}
 n\E\left[m_{1,n}^{k-j_2}m_{j_2,n}(m_{r,n}-\mu_r)\right]
 =
 {
 n\E\left(m_{1,n}^2 m_{r,n} m_{k-2,n}\right)-n\mu_r \E\left(m_{1,n}^2
 m_{k-2,n}\right)\to \sigma^2\mu_r\mu_{k-2} - \mu_r
 \sigma^2\mu_{k-2}=0,
 }
 \end{dmath*}
 and the proof is complete.
 \hfill\qed
 \end{pr}

 \begin{pr}{Proof of \Cref{theo.singular.limit}}
 Observe that $M_{k,n}-\mu_k=g_{k,k}(\bmk)-g_{k,k}(\bmuk)$;
 see in \Cref{sec.notation}.
 Also, $\surd{n}(\bbm_n-\bmuk)\distr \bWk$,
 where $\bWk=(W_1,\ldots,W_k)'\sim N(\bm{0}_k,\Sigk)$, see \eqref{clt}.
 Hence, \Cref{lem.delta2}
 applies to $\bX_n=\bmk$, provided \eqref{sing.var}
 is fulfilled for $\bmk$, i.e., provided that
 $n[\nabla g_{k,k}(\bmuk)]'(\bmk-\bmuk)\prob 0$.
 Because $\nabla g_{k,k}(\bmuk)=(-k\mu_{k-1},0,\ldots,0,1)'$,
 we get
 $[\nabla
 g_{k,k}(\bmuk)]'(\bmk-\bmuk)=-k\mu_{k-1}m_{1,n}+(m_{k,n}-\mu_k)$.
 Since $\E(m_{j,n})=\mu_j$ for all $n$ and $j$ we get
 $\E[-k\mu_{k-1}m_{1,n}+(m_{k,n}-\mu_k)]=0$. Also,
 \begin{dmath*}
 \Var[-k\mu_{k-1}m_{1,n}+(m_{k,n}-\mu_k)]
 =
 {
 k^2 \mu_{k-1}^2
 \Var(m_{1,n}) + \Var(m_{k,n})-2k\mu_{k-1}\Cov(m_{1,n},m_{k,n})
 }
 =
 k^2 \mu_{k-1}^2 \frac{\sigma^2}{n}
 + \frac{\mu_{2k}-\mu_k^2}{n}-2k\mu_{k-1}\frac{\mu_{k+1}}{n}
 =
 \frac{1}{n}\left(k^2 \mu_{k-1}^2 \sigma^2
 + \mu_{2k}-\mu_k^2-2k\mu_{k-1}\mu_{k+1}\right)
 =
 \frac{1}{n}\left[\mu_{2k}-\mu_k^2+k\mu_{k-1}
 \left(k\sigma^2\mu_{k-1}-2\mu_{k+1}\right)
 \right]
 \hiderel=
 \frac{v_k^2}{n}
 \hiderel=0,
 \end{dmath*}
 because $v_k^2=0$ by the assumed singularness.
 Therefore, $[\nabla g_{k,k}(\bmuk)]'(\bmk-\bmuk)=0$
 with probability one and, thus,
 $n[\nabla g_{k,k}(\bmuk)]'(\bmk-\bmuk)\prob 0$
 in a trivial sense. Now, a simple calculation,
 since $\nabla g_{k,k}(\bmuk)=(-k\mu_{k-1},0,\ldots,0,1)'$,
 shows that
 \[
 \BH_k(\bmuk)=
 \left(
 \begin{array}{cccccc}
 k(k-1)\mu_{k-2} & 0 & \cdots & 0 & -k & 0 \\
 0 & 0 & \cdots & 0 & 0 & 0 \\
 \vdots & \vdots & \ddots & \vdots & \vdots & \vdots \\
 0 & 0 & \cdots & 0 & 0 & 0 \\
 -k & 0 & \cdots & 0 & 0 & 0 \\
 0 & 0 & \cdots & 0 & 0 & 0
 \end{array}
 \right),
 \]
 i.e.,
 \[
 \BH_2(\bmu_2)=
 \left(
 \begin{array}{cc}
 2 & 0 \\
 0 & 0
 \end{array}
 \right),
 \quad
 \BH_3(\bmu_3)=
 \left(
 \begin{array}{ccc}
 0 & -3 & 0\\
 -3 & 0 & 0 \\
 0 & 0 & 0
 \end{array}
 \right),
 \quad
 \BH_4(\bmu_4)=
 \left(
 \begin{array}{cccc}
 12\sigma^2 & 0 & -4 & 0\\
 0 & 0 & 0 & 0\\
 -4 & 0 & 0 & 0 \\
 0 & 0 & 0 & 0
 \end{array}
 \right),
 \]
 e.tc. Applying \eqref{delta2}, we see that
 %the limiting distribution of
 $n(M_{k,n}-\mu_k)$ converges weakly to the
 distribution of
 $\frac{1}{2}\bW'_k\BH_k(\bmu_k)\bW_k
 =\frac{1}{2}k(k-1)\mu_{k-2} W_1^2 - k W_1 W_{k-1}$,
 while, by \eqref{clt}, the distribution of $(W_1,W_{k-1})'$
 is given by \eqref{sing.normal}.
 \hfill\qed
 \end{pr}
 }
 \begin{acknowledgements}
 We would like to thank H.\ Papageorgiou for helpful discussions.
 \end{acknowledgements}

\end{appendix}
\end{document}